\documentclass[a4paper,12pt,reqno]{amsart}
\usepackage{amsmath, amsfonts, amssymb, amsthm, mathrsfs, mathtools, mathalfa}
\usepackage{times}
\usepackage{color}
\usepackage{graphicx}
\usepackage{comment}
\usepackage{yfonts}
\usepackage{enumerate}
\usepackage{enumitem}
\usepackage{stmaryrd}
\usepackage{fancyhdr}
\usepackage{bm}
\usepackage[utf8]{inputenc}
\usepackage[pdfencoding=auto]{hyperref}
\usepackage{tikz,tikz-cd}
\usepackage[hang,flushmargin]{footmisc}
\usepackage{xr}
\usepackage{extarrows}
\usepackage[textwidth=2cm]{todonotes}
\usepackage{faktor}
\usepackage{amsmath}

\newtheorem{theorem}{Theorem}[section]
\newtheorem{lemma}[theorem]{Lemma}

\newtheorem{corollary}[theorem]{Corollary}
\newtheorem{proposition}[theorem]{Proposition}
\newtheorem{question}[theorem]{Question}

\theoremstyle{definition}
\newtheorem{defn}[theorem]{Definition}
\newtheorem{remark}[theorem]{Remark}

\def\deg{\mbox{deg}\,}

\newcommand*{\sbr}[1]{\scalebox{0.8}{$(#1)$}}
\newcommand*{\db}[1]{\llbracket #1\rrbracket}

\newcommand{\mc}{\mathcal}

\newcommand{\mb}{\mathbb}

\newcommand{\wh}{\widehat}
\newcommand{\wt}{\widetilde}

\newcommand{\id}{\mathrm{id}}

\DeclareMathOperator{\ab}{Z}

\DeclareMathOperator{\tran}{\Theta}

\DeclareMathOperator{\q}{c}
\DeclareMathOperator{\ns}{X}
\DeclareMathOperator{\nss}{Y}
\DeclareMathOperator{\co}{\circ\hspace{-0.02 cm}}
\DeclareMathOperator{\cu}{C}

\DeclareMathOperator{\pr}{pr}

\begin{document}

\vspace*{-2cm}

\title[On the inverse theorem for Gowers norms in abelian groups of bounded torsion]{On the inverse theorem for Gowers norms\\ in abelian groups of bounded torsion}

\author{Pablo Candela}
\address{Universidad Aut\'onoma de Madrid and ICMAT\\ 
Ciudad Universitaria de Cantoblanco\\ Madrid 28049\\ Spain}
\email{pablo.candela@uam.es}

\author{Diego Gonz\'alez-S\'anchez}
\address{MTA Alfr\'ed R\'enyi Institute of Mathematics, Re\'altanoda u. 13-15.\\
Budapest, Hungary, H-1053}
\email{diegogs@renyi.hu}

\author{Bal\'azs Szegedy}
\address{MTA Alfr\'ed R\'enyi Institute of Mathematics, Re\'altanoda u. 13-15.\\
Budapest, Hungary, H-1053}
\email{szegedyb@gmail.com}

\begin{abstract}
In recent work, Jamneshan, Shalom and Tao proved an inverse theorem for the Gowers $U^{k+1}$-norm on finite abelian groups of fixed torsion $m$, where the final correlating harmonic is a polynomial phase function of degree at most $C(k,m)$. They also posed a related central question, namely, whether the bound $C$ can be reduced to the optimal value $k$ for every $m$. We make progress on this question using nilspace theory. 
First we connect the question to the study of finite nilspaces whose structure groups have torsion $m$. Then we prove one of the main results of this paper: a primary decomposition theorem for finite nilspaces, extending the Sylow decomposition in group theory. Thus we give an analogue for nilspaces of an ergodic-theoretic Sylow decomposition in the aforementioned work of Jamneshan--Shalom--Tao. We deduce various consequences which illustrate the following general idea: the primary decomposition enables a reduction of higher-order Fourier analysis in the $m$-torsion setting to the case of abelian $p$-groups. These consequences include a positive answer to the question of Jamneshan--Shalom--Tao when $m$ is squarefree, and also a new relation between uniformity norms and certain generalized cut norms on products of abelian groups of coprime orders. Another main result in this paper is a positive answer to the above central question for the $U^3$-norm, proving that $C(2,m)=2$ for all $m$. Finally, we give a partial answer to the question for all $k$ and $m$, proving an inverse theorem involving extensions of polynomial phase functions which were introduced by the third-named author, known as \emph{projected} phase polynomials of degree $k$. A notable aspect is that this inverse theorem implies that of Jamneshan--Shalom--Tao, while involving projected phase polynomials of degree $k$, which are genuine obstructions to having small $U^{k+1}$-norm.
\end{abstract}
\maketitle

\section{Introduction}

\noindent Inverse theorems for Gowers norms are foundational results in higher-order Fourier analysis, their general idea being to characterize when a bounded function on an abelian group has a large Gowers norm. Since the original inverse theorem proved in \cite{GTZ} for functions on finite cyclic groups (or finite intervals in $\mb{Z}$), there have been several works deepening and extending this theory, proving inverse theorems for more general abelian groups, or focusing on special classes of abelian groups to establish more refined inverse theorems. We refer to \cite{CSinverse,GM2,J&T,MannersQIinv} for background on some of the more recent  developments in these directions.

This paper focuses on the case of the inverse theorem concerning abelian groups of \emph{bounded torsion} $m$ (or \emph{$m$-torsion abelian groups}) for some fixed positive integer $m$, that is, abelian groups $\ab$ such that every $x\in \ab$ satisfies $m x=0$. 

Our starting point and main inspiration in this direction is the recent work of Jamneshan, Shalom and Tao \cite{JST-tot-dis}, and in particular the following theorem (see \cite[Theorem 1.12]{JST-tot-dis}). Recall that given abelian groups $\ab,\ab'$, given a map $P:\ab\to\ab'$ and any $h\in \ab$, we may take the discrete derivative $\partial_h P:\ab\to\ab'$ defined by $\partial_h P(x)=P(x+h)-P(x)$, and then we say that $P$ is \emph{polynomial of degree at most $k$} if applying discrete derivatives $k+1$ times to $P$ we obtain the 0 map, i.e.\ if $\partial_{h_1}\cdots\partial_{h_{k+1}}(P)(x)=0$ for all $x,h_1,\ldots,h_{k+1}\in \ab$.

\begin{theorem}\label{thm:JSTbt}
Let $k,m$ be positive integers and let $\delta > 0$. Then there exist constants $\varepsilon = \varepsilon(\delta,k,m)>0$ and $C = C(k, m)>0$ such that for every finite $m$-torsion abelian group $\ab$ and every 1-bounded function $f: \ab\to \mb{C}$ with $\| f \|_{U^{k+1}} > \delta$, there exists a polynomial $P:\ab\to\mb{R}/\mb{Z}$ of degree at most $C$ such that
$|\mb{E}_{x\in \ab} f (x)e(-P(x))| >\varepsilon$.
\end{theorem}
\noindent An ancestor of (and motivator for) Theorem \ref{thm:JSTbt} is the well-known result in the case where $m$ is prime, namely the inverse theorem for vector spaces $\mb{F}_p^n$ from \cite[Theorems 1.9 and 1.10]{T&Z-High} and \cite[Theorem 1.10]{T&Z-Low}. This is generally viewed as a qualitatively satisfactory inverse theorem in that case, and Theorem \ref{thm:JSTbt} contributes to an extension of this result into the realm of bounded-torsion groups. However, in the case of $\mb{F}_p^n$, the degree bound for the polynomial $P$ is guaranteed to be the optimal one, namely $k$ (rather than some constant $C$ possibly larger than $k$ as in Theorem \ref{thm:JSTbt}). This motivated Question 1.9 in \cite{JST-tot-dis}, which we recall here for convenience.
\begin{question}[See Question 1.9 in \cite{JST-tot-dis}]\label{Q:JST} 
Can one take the constant $C(k,m)$ in Theorem \ref{thm:JSTbt} always to be $k$?
\end{question}

\noindent In this paper we make progress on this question using nilspace theory (see \cite{CamSzeg,Cand:Notes1,Cand:Notes2,GMV1,GMV2,GMV3}). 

Our first step is to apply the general inverse theorem for compact abelian groups \cite[Theorem 5.2]{CSinverse} and observe that, when this is restricted to the present case of $m$-torsion abelian groups, the resulting compact nilspace (underlying the nilspace polynomial) in \cite[Theorem 5.2]{CSinverse} must be what we will call an \emph{$m$-torsion nilspace}, that is, a nilspace whose structure groups are all of torsion $m$. We prove this in Section \ref{sec:nil-bnd-tor} by a simple adaptation of an argument that had already yielded the finite-field case of this result in \cite{CGSS-p-hom}; see Theorem \ref{thm:nil-for-m-exp}. This result can be compared with \cite[Theorem 1.4]{JST-tot-dis}, which is an analogue in ergodic theory.

The first step above yields an approach to Question \ref{Q:JST} consisting in analyzing $m$-torsion nilspaces of finite rank. These nilspaces are finite (see the proof of \cite[Proposition 2.3]{CGSS-p-hom}). Focusing on such finite nilspaces, we then prove one of the main results of this paper, namely a \emph{primary decomposition} (or \emph{Sylow decomposition}) for finite nilspaces, generalizing the classical primary decomposition of finite abelian groups. To formulate this, we introduce the following class of nilspaces, which extends the class of $p$-homogeneous nilspaces from \cite{CGSS-p-hom}.
\begin{defn}\label{def:pnil}
Let $p$ be a prime. We say that a $k$-step nilspace $\ns$ is a \emph{$p$-nilspace} if every structure group of $\ns$ is a $p$-group. (For the notion of structure groups see \cite[Theorem 3.2.19]{Cand:Notes1}.)
\end{defn}
For any integer $N$, we shall write $\mc{P}(N)$ for the set of primes dividing $N$.
\begin{theorem}[Primary decomposition of finite nilspaces]\label{thm:p-sylow-intro}
Let $\ns$ be a finite $k$-step nilspace. Then $\ns$ is isomorphic to a product nilspace $\prod_{p\in \mc{P}(|\ns|)} \ns_p$ where $\ns_p$ is a $k$-step $p$-nilspace for each $p\in \mc{P}(|\ns|)$. Moreover, this decomposition is unique up to isomorphism.
\end{theorem}
\noindent This is a nilspace-theoretic analogue of the recent ergodic-theoretic Sylow decomposition theorem \cite[Theorem 2.3]{JST-tot-dis}. The proofs of these two results are different, however, and we did not find a simple deduction of the former  from the latter (see Remark \ref{rem:JSTimplication}).

We prove Theorem \ref{thm:p-sylow-intro} in Section \ref{sec:primdecomp}. The main consequences of Theorem \ref{thm:p-sylow-intro} that we obtain can be described, generally speaking, as various ways to reduce higher-order Fourier analysis on bounded-torsion abelian groups to the case of abelian $p$-groups. We now substantiate this idea with the formal statements of these consequences, which we shall prove in Subsection \ref{Subsec:PrimDecompConseq}. 

The first consequence is the following version of the inverse theorem in the $m$-torsion setting. Recall from \cite[Definition 1.3]{CSinverse} the notion of a \emph{nilspace polynomial of degree $k$} on a finite abelian group $\ab$, namely a function $F\co\phi:\ab\to\mb{C}$ where $\phi$ is a morphism from $\mc{D}_1(\ab)$ to some $k$-step compact nilspace $\ns$ of finite rank (we denote the set of such morphisms by $\hom(\mc{D}_1(\ab),\ns)$; see \cite{Cand:Notes1}), and where $F$ is a 1-bounded function $\ns\to\mb{C}$. We say that $F\co\phi$ has \emph{complexity at most} $M$ if  $\ns$ has complexity at most $M$ as per \cite[Definition 1.2]{CSinverse}, and that $F\co\phi$ is a \emph{$p$-nilspace} (resp. \emph{$m$-torsion nilspace}) \emph{polynomial} if $\ns$ is a $p$-nilspace (resp. $m$-torsion nilspace).
\begin{theorem}\label{thm:invreduc-intro}
For every $k,m\in \mb{N}$ and $\delta > 0$, there exist  $\varepsilon>0$ and $M>0$ such that the following holds. Let $\ab$ be a finite $m$-torsion abelian group with primary decomposition $\ab=\bigoplus_{p\in \mc{P}} \ab_p$ for $\mc{P}:=\mc{P}(|\ab|)$, and let $f: \ab\to \mb{C}$ be a 1-bounded function with $\| f \|_{U^{k+1}} > \delta$. Then for each $p\in \mc{P}$ there is a $p$-nilspace polynomial $F_p\co \phi_p:\ab_p\to\mb{C}$ of complexity at most $M$ such that $|\mb{E}_{z\in \ab} f(z) \prod_{p\in \mc{P}} F_p\co \phi_p(z_p)| \geq \varepsilon$ \textup{(}where each $z\in\ab$ is written $(z_p)_{p\in \mc{P}}$ for $z_p\in \ab_p$\textup{)}.
\end{theorem}
\noindent We shall in fact prove a refined version of this result which adds equidistribution control (in the form of the \emph{balance} property) to the nilspace polynomials; see Theorem \ref{thm:invreduc}. Using these results, we reduce Question \ref{Q:JST} to the following question, which focuses on $p$-groups.
\begin{question}\label{mainQ:p-case}
Let $p$ be a prime, let $\ab$ be a finite abelian $p$-group, and let $\phi$ be a morphism from $\mc{D}_1 (\ab)$ to some $k$-step finite  $p$-nilspace $\ns$. If $\phi$ is sufficiently balanced, must there  exist a $k$-step finite abelian \emph{group} nilspace\footnote{i.e.\ $\nss$ is a finite abelian group equipped with the Host--Kra cubes associated with a filtration of degree $k$ on $\nss$.} $\nss$, of cardinality $|\nss|=O_{|\ns|,k,p}(1)$, a morphism $g:\mc{D}_1(\ab)\to\nss$, and a function $h:\nss\to\ns$ \textup{(}not required to be a morphism\textup{)}, such that $\phi=h\co g$?
\end{question}
\noindent A positive answer to Question \ref{mainQ:p-case} would give an extension for $p$-groups of known results for vector spaces $\mb{F}_p^n$, such as \cite[Theorem 1.7]{CGSS-p-hom} (the balance assumption in Question \ref{mainQ:p-case} may be unnecessary). The reduction of Question \ref{Q:JST} to Question \ref{mainQ:p-case} is detailed in Corollary \ref{cor:Q1.9reduc}. Note that, as a particular consequence of this reduction, we can answer  Question \ref{Q:JST} positively in the \emph{squarefree case}, i.e.\ when $m$ is the product of distinct primes; see Corollary \ref{cor:inv-disct-primes}.

Another consequence of Theorem \ref{thm:p-sylow-intro} is a surprising connection between the uniformity norms and generalizations of the \emph{cut norm} on direct sums of abelian groups of coprime orders. Let us recall the definition of these cut norms from \cite[Definition 4.5]{Cas} (see also \cite{ConLee}).
\begin{defn}[$(n,d)$-cut norm]\label{def:cn}
Let $A_1,\ldots,A_n$ be finite sets and let $d\in [1, n-1]$ be an integer. For any $f:\prod_{i=1}^n A_i\to \mb{C}$, the \emph{$(n,d)$-cut norm} of $f$ is defined as follows:
\begin{eqnarray*}
    & \|f \|_{\Box_d^n}:=\max_{u_B:\prod_{i\in B} A_i \to \mb{D}, \,\forall B\in \binom{[n]}{d}} \left|\mb{E}_{\underline{a}\in \prod_{i=1}^n A_i}  f(\underline{a})\prod_{B\in \binom{[n]}{d}} \overline{u_B(\underline{a}_B)}\right|, &
\end{eqnarray*}
where $\mb{D}$ is the unit disc in $\mb{C}$ and $\underline{a}_B$ is the projection of $\underline{a}$ to the coordinates indexed by $B$.
\end{defn}
\noindent These norms are polynomially equivalent to averages of $f$ over certain hypergraphs (see \cite[Lemma 5.8]{ConLee} or \cite[Theorem 4.11]{Cas}). In particular the classical cut norm (i.e.\ the $(2,1)$-cut norm) is polynomially equivalent to the 4-cycle norm (or Gowers 2-box norm) $\|f \|_{\Box(A\times B)}^{4}=\mb{E}_{a_1,a_2\in A,b_1,b_2\in B} f(a_1,b_1)\overline{f(a_2,b_1)}\overline{f(a_1,b_2)}f(a_2,b_2)$. The connection in question is the following.

\begin{corollary}\label{cor:box-norm-estimate}
Let $k\in \mb{N}$ and let $m_1,\ldots,m_n\in \mb{N}$ be pairwise coprime. Then for any $\delta>0$ there exists $\varepsilon=\varepsilon(\delta,m_1,\ldots,m_n,k)>0$ such that the following holds. For each $i\in [n]$ let $\ab_i$ be a finite abelian group of torsion $m_i$, and let $f:\prod_{i=1}^n \ab_i\to \mb{C}$ be a 1-bounded function satisfying $\|f\|_{U^{k+1}}>\delta$. Then $\|f \|_{\Box^n_1}>\varepsilon$. 
\end{corollary}
\noindent In particular, if $\ab_1,\ab_2$ are finite abelian groups of coprime orders, then for every $k$,   smallness of $f:\ab_1\times\ab_2\to\mb{D}$ in the 4-cycle norm $\|\cdot \|_{\Box(\ab_1\times\ab_2)}$ implies smallness of $\|f\|_{U^{k+1}(\ab_1\times\ab_2)}$. Note that this result fails without the coprimality condition. Indeed, letting $\ab_1=\ab_2=\mb{F}_2^\ell$, the function $f: (x,y)\mapsto e(\frac{1}{2}\sum_{i=1}^\ell x_i y_i)$ on $\ab_1\times \ab_2$ satisfies $\|f\|_{U^3}=1$ but $\|f\|_{\Box(\ab_1\times \ab_2)}= 2^{-\ell/4}$.

\medskip

\noindent After proving the above results, we focus on the first non-trivial case of Question \ref{Q:JST}, namely for $k=2$ (the case $k=1$ is easily proved using standard Fourier analysis). Our main result here is the following affirmative answer in this case, proved in Section \ref{sec:proof-of-main}.
\begin{theorem}\label{thm:mainU3}
For every $m\in\mb{N}$ and $\delta > 0$, there exists $\varepsilon = \varepsilon(m,\delta)$ such that for every finite $m$-torsion abelian group $\ab$ and any 1-bounded function $f: \ab\to \mb{C}$ with $\| f \|_{U^3} > \delta$, there exists a polynomial $P:\ab\to\mb{R}/\mb{Z}$ of degree at most 2 such that
$|\mb{E}_{x\in \ab} f (x)e(-P(x))| >\varepsilon$.
\end{theorem}
\noindent The main new ingredient in our proof of this theorem is an algebraic result which, roughly speaking, enables us to replace a given bounded-index subgroup $H$ of $\ab$ by a bounded-index subgroup $H'\leq H$ with the added property that $H'$ has an additive complement in $\ab$; see Proposition \ref{prop:comp1}. We obtain Theorem \ref{thm:mainU3} basically by combining this ingredient with Theorem \ref{thm:invreduc-intro}. See also Remark \ref{rem:AltPfs} for another way to use the above ingredient to obtain Theorem \ref{thm:mainU3}, observed a posteriori by Jamneshan, Shalom and Tao (in personal communication).

For $k>2$, we make progress on Question \ref{Q:JST} by proving an inverse theorem for the $U^{k+1}$-norm on $m$-torsion groups, with correlating functions being polynomial objects of \emph{optimal} degree $k$, which are not quite polynomial phase functions, but somewhat more general functions called \emph{projected phase polynomials}. These  were already seen to play a useful role in this inverse theory in the unpublished work \cite{SzegFin}. We recall their definition here (see \cite[Definition 1.2]{SzegFin}). 
\begin{defn}\label{def:propolyphase}
Let $\ab$ be a finite abelian group and let $k\in \mb{N}$. A \emph{projected phase polynomial of degree $k$} on $\ab$ is a 1-bounded function $\phi_{*\tau}:\ab\to\mb{C}$ of the following form. There is a finite abelian group $\ab'$, a surjective homomorphism $\tau:\ab'\to \ab$, and a polynomial phase function $\phi:\ab'\to\mb{C}$ of degree at most $k$, such that $\phi_{*\tau}(x) = \mb{E}_{y\in \tau^{-1}(x)}\phi(y)$ for every $x\in\ab$.  
\end{defn}
\noindent Thus $\phi_{*\tau}\co\tau =\mb{E}( \phi|\tau)$, with the expectation taken relative to the partition of $\ab'$ induced by $\tau$. (The notation $\phi_{*\tau}$ is inspired by the one for the image of a measure under a measurable map $\tau$.)
\begin{defn}
Let $\phi_{*\tau}$ be a projected phase polynomial of degree $k$ on some finite abelian group $\ab$, for some surjective homomorphism $\tau:\ab'\to \ab$. If the torsions of $\ab$ and $\ab'$ are $m$ and $m'$ respectively, then we say that $\phi_{*\tau}$ has \emph{torsion}  $(m,m')$. We say that $\phi_{*\tau}$ is \emph{rank-preserving} if $\textrm{rk}(\ab')=\textrm{rk}(\ab)$ (where the \emph{rank} $\textrm{rk}(\ab)$ is the minimum cardinality of a generating set for $\ab$).
\end{defn}
\begin{theorem}\label{thm:invboundedtor}
Let $k,m$ be positive integers and let $\delta>0$. Then there exists $\gamma=\gamma(k)\in \mb{N}$ and $\varepsilon=\varepsilon(\delta,k,m)>0$ such that the following holds. For any $m$-torsion abelian group $\ab$ and any 1-bounded function $f:\ab\to \mb{C}$ with $\|f\|_{U^{k+1}}\ge \delta$, there exists a rank-preserving projected phase polynomial $\phi_{*\tau}$ of degree $k$ and torsion $(m,m^\gamma)$ on $\ab$ such that $|\mb{E}_{x\in\ab} f(x) \overline{\phi_{*\tau}(x)}|\ge \varepsilon$. 
\end{theorem}
\noindent We prove this in Section \ref{sec:prophase}, where we also treat two other aspects that further illustrate the relevance of projected phase polynomials for the $m$-torsion setting. Firstly, we show that projected phase polynomials of degree $k$ are \emph{genuine} obstructions to Gowers uniformity of degree $k$, thus establishing a tightness in Theorem \ref{thm:invboundedtor} which is not present in $U^{k+1}$-inverse theorems that use polynomials of degree possibly larger than $k$; see Proposition \ref{prop:propolyobstruct}. Secondly, we prove that Theorem \ref{thm:invboundedtor} implies Theorem \ref{thm:JSTbt}, by showing that a projected phase polynomial of degree $k$ can be expressed as an average of genuine phase polynomials of some degree $C(k,m)$.

The paper ends with questions suggesting routes for further progress on \cite[Question 1.9]{JST-tot-dis}.

\vspace*{0.2cm}

\noindent \textbf{Acknowledgements}. All authors used funding from project PID2020-113350GB-I00 of Spain’s MICINN.
The second-named author received funding from projects KPP 133921 and Momentum
(Lend\"ulet) 30003 of the Hungarian Government. The research was also supported partially
by the NKFIH ``\'Elvonal” KKP 133921 grant and by the Hungarian Ministry of
Innovation and Technology NRDI Office in the framework of the Artificial Intelligence
National Laboratory Program. We thank Asgar Jamneshan, Or Shalom and Terence Tao for useful comments on the first version of this paper.

\section{An inverse theorem with bounded-torsion nilspaces}\label{sec:nil-bnd-tor}

\noindent As mentioned in the introduction, we say that a nilspace $\ns$ is \emph{$m$-torsion} if every structure group of $\ns$ is an $m$-torsion group. 

In this section we record the fact that when the general inverse theorem \cite[Theorem 5.2]{CSinverse} is specialized to the case of $m$-torsion abelian groups, then the resulting nilspace is an $m$-torsion compact (hence finite) nilspace. To this end we gather the following two main ingredients. 

The first ingredient is \cite[Theorem 5.2]{CSinverse}, which we state below for convenience. Recall from \cite[Definition 5.1]{CSinverse} the notion of \emph{balance} for a morphism $\phi:\ns\to \nss$ between two compact $k$-step nilspaces $\ns,\nss$. Recall also the notion of \emph{complexity} of a compact finite-rank (\textsc{cfr}) nilspace $\nss$, denoted Comp$(\nss)$, underlying the complexity of a nilspace polynomial (see \cite[Definitions 1.2 and 1.3]{CSinverse}). 
\begin{theorem}\label{thm:gen-inverse}
Let $k\in \mb{N}$, and let $b:\mb{R}_{>0}\to \mb{R}_{>0}$ be an arbitrary function. For every $\delta\in (0,1]$ there is $M>0$ such that for every \textsc{cfr} coset nilspace $\ns$, and every 1-bounded Borel function $f:\ns\to \mb{C}$ such that $\|f\|_{U^{k+1}}\geq \delta$, for some $r\leq M$ there is a $b(r)$-balanced 1-bounded nilspace polynomial $F\co\phi$ of degree $k$ and complexity at most $r$ such that $\langle f, F\co\phi\rangle \geq \delta^{2^{k+1}}/2$.
\end{theorem}
\noindent The second ingredient is a generalization to the $m$-torsion setting of what was already a key ingredient in \cite{CGSS-p-hom}, namely \cite[Proposition 2.3]{CGSS-p-hom}. The latter result states that for a prime $p$, if a morphism $\phi:\mc{D}_1(\mb{F}_p^n)\to \ns$ is sufficiently balanced (depending on $\ns$, a particular metric $d$ on $\ns$, and $p$) then every structure group of $\ns$ is a $p$-torsion group (i.e.\ an elementary abelian $p$-group). The generalization that we obtain for the $m$-torsion case is the following.
\begin{proposition}\label{prop:b-bal-m-bounded-finite}
Let $\nss$ be a $k$-step \textsc{cfr} nilspace, let $d$ be a metric generating the topology on $\nss$, and let $m\in \mb{N}$. There exists $b=b(\nss,d,m)>0$ such that the following holds: if for some $m$-torsion abelian group $\ab$ there is a $b$-balanced morphism $\varphi:\mc{D}_1(\ab)\to \nss$, then every structure group of $\nss$ is a finite abelian $m$-torsion group, and in particular $\nss$ is finite.
\end{proposition}

\begin{proof}
The argument is essentially the same as the proof of \cite[Proposition 2.3]{CGSS-p-hom}. In fact, it suffices to replace $\mb{Z}_p^D$, $\mb{Z}_p^k$ and $p$ in that proof by $\ab$, $\mb{Z}_m^k$ and $m$ respectively. (Here and elsewhere in this paper we denote by $\mb{Z}_N$ the cyclic group of integers with addition modulo $N$.)
\end{proof}

Combining these two ingredients we obtain the main result of this section.

\begin{theorem}\label{thm:nil-for-m-exp}
For every $k,m\in \mb{N}$ and $\delta>0$ there exists $C>0$ such that the following holds. Let $\ab$ be a finite abelian $m$-torsion group, and let $f:\ab\to \mb{C}$ be a $1$-bounded function with $\|f\|_{U^{k+1}}\ge \delta$. Then there is a finite $m$-torsion nilspace $\ns$ of cardinality $|\ns|\le C$, a morphism $\phi:\mc{D}_1(\ab)\to \ns$, and a $1$-bounded function $F:\ns\to \mb{C}$, such that $\langle f,F\co\phi\rangle \ge \tfrac{1}{2}\delta^{2^{k+1}}$. Moreover, for any fixed function $D:\mb{R}_{>0}\to \mb{R}_{>0}$ we can further assume that there exists $M=M(D,\delta)>0$ such that $\ns$ has complexity at most $r\le M$ and $\phi$ is $D(r)$-balanced.
\end{theorem}

\begin{proof}
We apply Theorem \ref{thm:gen-inverse} with $b$ chosen as follows. For each $r$, we let $0<b(r)<\min\{b'(\nss\sbr{i},d_{\nss\sbr{i}},m):i=\textrm{Comp}(\nss\sbr{i})\le r\}$ where $b'(\nss\sbr{i},d_{\nss\sbr{i}},m)$ is given by Proposition \ref{prop:b-bal-m-bounded-finite} applied to the nilspace $\nss\sbr{i}$ in our fixed complexity notion. Therefore, the nilspace $\ns$  given by Theorem \ref{thm:gen-inverse} equals $\nss\sbr{i}$ for some $i\le M$ and has all its structure groups of torsion $m$ by Proposition \ref{prop:b-bal-m-bounded-finite}. Moreover, we can bound the size of $\ns$ as follows. Let $C:=\max\{|\nss\sbr{i}|:i= \textrm{Comp}(\nss\sbr{i})\le M \text{ and }\nss\sbr{i} \text{ is finite.}\}$. Clearly $|\ns|\le C$.

For the last part of the theorem, we can further apply Theorem \ref{thm:gen-inverse} with $\min(b(r),D(r))$ to deduce that $\ns$ has complexity at most $r$ and $\phi$ is $D(r)$-balanced.
\end{proof}

\section{The primary decomposition of finite nilspaces}\label{sec:primdecomp}

\noindent The fundamental theorem of finite abelian groups yields the decomposition of any such group $G$ as the direct sum of the $p$-groups $G_p:=\{x\in G:\textrm{the order of $x$ is a power of $p$}\}$ where $p$ runs through $\mc{P}(|G|)$ (also known as the \emph{primary decomposition} of $G$) \cite[2.1.6 Theorem]{KSfintie groups}. This phenomenon extends into the nilpotent setting, with the decomposition of any finite nilpotent group as the direct product of its Sylow $p$-subgroups \cite[5.1.4 Theorem]{KSfintie groups}. Recently, it was shown in \cite[Theorem 2.3]{JST-tot-dis} that a similar phenomenon occurs in ergodic theory when we study the $k$-th order Host-Kra factor of an ergodic  $\Gamma$-system where the $\Gamma$ is $m$-torsion. In this section we prove Theorem \ref{thm:p-sylow-intro}, establishing such a decomposition for every finite nilspace. Recall from Definition \ref{def:pnil} the notion of a $p$-nilspace. We shall first prove the following theorem yielding the decomposition, leaving the proof of uniqueness for the end of this section.
\begin{theorem}\label{thm:p-sep}
Let $\ns$ be a finite $k$-step nilspace. Then for every prime $p\in\mc{P}(|\ns|)$ there is a $k$-step $p$-nilspace $\ns_p$ such that we have the following isomorphism of nilspaces:
\begin{equation}\label{eq:p-sep}
\ns\cong \prod_{p\in \mc{P}(|\ns|)} \ns_p.
\end{equation}
\end{theorem}
\noindent Since the order of a $p$-nilspace is a power of $p$, we have in particular that if $|\ns|=\prod_{p\in \mc{P}(|\ns|)} p^{\alpha_p}$ is the prime factorization of $|\ns|$ then each $p$-nilspace $\ns_p$ in \eqref{eq:p-sep} has cardinality $|\ns_p|=p^{\alpha_p}$.

The proof of Theorem \ref{thm:p-sep} argues by induction on the step $k$. The central statement for the induction is the following. Recall the notion of degree-$k$ extensions from \cite[Definition 3.3.13]{Cand:Notes1}.

\begin{proposition}\label{prop:psepkeyprop}
Let $\nss_1,\nss_2$ be finite nilspaces and let $\ab$ be a finite abelian group such that $|\ab|$ and $|\nss_1|$ are coprime. Let $\ns$ be a finite nilspace which is a degree-$k$ extension of $\nss_1\times\nss_2$ by $\ab$. Then $\ns$ is isomorphic to $\nss_1\times \nss_2'$ where $\nss_2'$ is a degree-$k$ extension of $\nss_2$ by $\ab$.
\end{proposition}

\begin{proof}[Proof of Theorem \ref{thm:p-sep} using Proposition \ref{prop:psepkeyprop}]
We argue by induction on $k$, using also a sub-induction on the number of distinct prime factors of the $k$-th structure group of $\ns$. For $k=1$ the nilspace $\ns$ is (an affine version of) a finite abelian group, so the result follows from the primary decomposition of finite abelian groups recalled above. We can therefore assume that $k\geq 2$ and that the theorem holds for $(k-1)$-step nilspaces. Let $|\ns|=p_1^{\alpha_1}\cdots p_m^{\alpha_m}$ be the prime factorization of $|\ns|$, for positive integers $\alpha_i$. Let $\ab_k$ be the last structure group of $\ns$, and note that since $|\ab_k|$ divides $|\ns|$, we have $|\ab_k|=p_1^{\beta_1}p_2^{\beta_2}\dots p_m^{\beta_m}$ for integers $\beta_i\geq 0$. If all $\beta_i$ are 0, then $|\ab_k|=1$ and $\ns$ is in fact of step $k-1$, so the induction hypothesis gives the result. We can therefore assume (permuting factors if needed) that $\beta_m>0$ and that the statement holds in the case where $|\ab_k|$ has at most $m-1$ distinct prime factors. Let us write $\ab_k=\ab_{k,1}\times \dots\times \ab_{k,m}$ with $|\ab_{k,i}|=p_i^{\beta_i}$. Let $\ns'$ be the nilspace $\ns/\ab_{k,m}$ (this is a well-defined nilspace by \cite[Proposition A.19]{CGSS-p-hom}). 
By our induction and sub-induction hypotheses, we have a primary decomposition $\ns'\cong \prod_{i=1}^r \ns'_{q_i}$ where the $q_i$ are the prime divisors of $|\ns'|$, and $\ns'_{q_i}$ is a $q_i$-nilspace for each $i$. Since $|\ns'|=|\ns|/|\ab_{k,m}|$, by permuting the factors of $\ns'$ if necessary, we can assume without loss of generality that $r=m$, that $q_i=p_i$ for each $i\in [m]$, and that $|\ns'|=p_1^{\alpha_1}\cdots p_{m-1}^{\alpha_{m-1}}p_m^{\alpha_m'}$, where $\alpha_m'=\alpha_m-\beta_m\geq 0$. Hence $\ns$ is a degree-$k$ extension of $\ns'$ by $\ab_{k,m}$. By Proposition \ref{prop:psepkeyprop} applied with $\nss_1=\prod_{i=1}^{m-1} \ns'_{p_i}$, $\nss_2=\ns_{p_m}'$, and $\ab=\ab_{k,m}$, we deduce that $\ns\cong (\prod_{i=1}^{m-1} \ns'_{p_i})\times \ns_{p_m}''$, where $\ns_{p_m}''$ is a degree-$k$ extension of $\ns_{p_m}'$ by $\ab_{k,m}$ (and therefore a $p_m$-nilspace). This completes the induction and Theorem \ref{thm:p-sep} follows.
\end{proof}
\noindent We can now focus on proving Proposition \ref{prop:psepkeyprop}. To this end we will work with nilspace cocycles (see \cite[\S 3.3.3]{Cand:Notes1}). By basic nilspace theory we know that the extension is given by a cocycle $\rho:\cu^{k+1}(\nss_1\times\nss_2)\to \ab$. The main result which will imply Proposition \ref{prop:psepkeyprop} is the following, yielding a cocycle decomposition of the form $\rho=\kappa+\tau$, where $\tau$ is a coboundary, and $\kappa:\cu^{k+1}(\nss_1\times\nss_2)\to \ab$ factors through the 2-nd component projection on $\nss_1\times \nss_2$.
\begin{proposition}\label{prop:cocycledecomp}
Let $\nss_1,\nss_2$ be finite nilspaces and let $\ab$ be a finite abelian group such that $|\ab|$ and $|\nss_1|$ are coprime. Let $\ns$ be a finite nilspace which is a degree-$k$ extension of $\nss_1\times\nss_2$ by $\ab$, and let $\rho:\cu^{k+1}(\nss_1\times\nss_2)\to \ab$ be the associated cocycle. Then there is a cocycle $\kappa_2:\cu^{k+1}(\nss_2)\to \ab$ such that, letting $\pi_2:\nss_1\times \nss_2\to\nss_2$ be the projection $(y_1,y_2)\mapsto y_2$, and letting $\kappa=\kappa_2\co\pi_2^{\{0,1\}^{k+1}}:\q\mapsto  \kappa_2(\pi_2\co\q)$, we have that $\rho-\kappa$ is a coboundary. 
\end{proposition}
\begin{proof}[Proof of Proposition \ref{prop:psepkeyprop} using Proposition \ref{prop:cocycledecomp}]
Since $\rho-\kappa$ is a coboundary, by the basic theory of cocycles (see \cite[Corollary  3.3.29]{Cand:Notes1})  we have that $\ns$ is isomorphic to the nilspace which is the extension of $\nss_1\times\nss_2$ by $\ab$ associated with the cocycle $\kappa$ (as per \cite[Definition 3.3.24 and Proposition 3.3.26]{Cand:Notes1}). Given the form of $\kappa$, this extension of $\nss_1\times\nss_2$ by $\ab$ associated with $\kappa$ is isomorphic to $\nss_1\times\nss_2'$ where $\nss_2'$ is the extension of $\nss_2$ by $\ab$ with associated cocycle $\kappa_2$. Indeed, recall that by \cite[Definition 3.3.24]{Cand:Notes1} the extension of $\nss_1\times \nss_2$ by the cocycle $\kappa$ is defined as $M_{\nss_1\times \nss_2}:=\bigcup_{(y_1,y_2)\in \nss_1\times \nss_2}\{\kappa_{(y_1,y_2)}+z:z\in \ab\}$ and the cubes given by \cite[Definition 3.3.25]{Cand:Notes1}. Similarly, the extension of $\nss_2$ by $\kappa_2$ is the nilspace $M_{\nss_2}:=\bigcup_{y_2\in \nss_2}\{{\kappa_2}_{y_2}+z:z\in \ab\}$. A simple calculation shows that the map $\nss_1\times M_{\nss_2} \to M_{\nss_1\times \nss_2}$ defined as $(y_1,{\kappa_2}_{y_2}+z)\mapsto \kappa_{(y_1,y_2)}+z$ is a nilspace isomorphism.
\end{proof}
\noindent We thus come to the core of the proof of Theorem \ref{thm:p-sep}, which consists in confirming the cocycle decomposition in Proposition \ref{prop:cocycledecomp}. We shall ``average out" the $\nss_1$ variables of the cocycle $\rho$ to obtain a cocycle $\kappa$ with the claimed property of factoring through $\pi_2$, and then prove that $\rho-\kappa$ is a coboundary. To achieve this we need to develop certain useful averaging operators. To this end we shall use the following facts. 

\begin{lemma}\label{lem:coprime-hom}
Let $\ns$ be a $k$-step finite nilspace and let $\ab$ be a finite abelian group such that $|\ab|, |\ns|$ are coprime. Let $n\in \mb{Z}_{\ge 0}$, let $P\subset \db{n}:=\{0,1\}^n$ such that $P$ has the extension property in $\db{n}$ \textup{(}see \cite[Definition 3.1.3]{Cand:Notes1}\textup{)}, and let $P'\subset P$ have the extension property in $P$. Then for every morphism $f:P'\to \ns$, the cardinalities $|\hom_f(P,\ns)|$ and $|\ab|$ are coprime.\footnote{See \cite[Definition 3.3.10]{Cand:Notes1} for the definition of $\hom_f(P,\ns)$.}
\end{lemma}

\begin{proof}
By \cite[Lemma 3.3.11]{Cand:Notes1}, the set of restricted morphisms $\hom_f(P,\ns)$ is a $k$-fold abelian bundle whose structure groups are $\hom_{P'\to 0}(P,\mc{D}_i(\ab_i(\ns)))$ for $i\in[k]$. Hence, it suffices to prove that for every $i\in[k]$ the cardinalities $|\hom_{P'\to 0}(P,\mc{D}_i(\ab_i(\ns)))|$ and $|\ab|$ are coprime. Let $\phi_1:\hom\big(\db{n},\mc{D}_i(\ab_i(\ns))\big)\to \hom\big(P,\mc{D}_i(\ab_i(\ns))\big)$ be the restriction map, which is a surjective homomorphism by the extension property. Since the order of $\hom(\db{n},\mc{D}_i(\ab_i(\ns)))=\cu^n(\mc{D}_i(\ab_i(\ns)))$ equals $|\ab_i(\ns)|^{t}$ for some $t\in \mb{N}$, the order of $\hom(P,\mc{D}_i(\ab_i(\ns)))$ is also coprime with $|\ab|$. Let $\phi_2:\hom(P,\mc{D}_i(\ab_i(\ns)))\to \hom(P',\mc{D}_i(\ab_i(\ns)))$ be also the restriction map, which is also a surjective homomorphism. Now note that $\hom_{P'\to 0}(P,\mc{D}_i(\ab_i(\ns)))=\ker(\phi_2)$, and thus its order is coprime with $|\ab|$.\end{proof}
\begin{remark}\label{rem:avZ}
From now on, for any function $f$ defined on a finite set $X$ and taking values in a finite abelian group $\ab$ such that $|X|$ and $|\ab|$ are coprime, the averaging notation $\mb{E}_{x\in X} f(x)$ will always denote the unique element $z\in \ab$ such that $|X|\, z=\sum_{x\in X} f(x)$.
\end{remark}
\begin{corollary}\label{cor:mes-presev-with-torsion}
Let $\ns$ be a $k$-step finite nilspace and let $\ab$ be a finite abelian group of order coprime with $|\ns|$. Let $n\in \mb{Z}_{\ge 0}$, let $P\subset \db{n}$ be a set with the extension property in $\db{n}$ and let $P_1,P_2$ be a \emph{good pair}\footnote{See \cite[\S 3.1.1]{Cand:Notes1} and \cite[Definition 2.2.13]{Cand:Notes2}.} of subsets of $P$. Let $f:P_1\to \ns$ be a morphism and let $\psi:\hom_f(P,\ns)\to \hom_{f|_{P_1\cap P_2}}(P_2,\ns)$ be the restriction map. Then for every function $h:\hom_{f|_{P_1\cap P_2}}(P_2,\ns)\to \ab$ we have $
\mb{E}_{x\in \hom_{f|_{P_1\cap P_2}}(P_2,\ns)}\, h(x) = \mb{E}_{x\in \hom_f(P,\ns)} h(\psi(x))$.
\end{corollary}
\begin{proof}
By Lemma \ref{lem:coprime-hom} we know that both $|\hom_{f|_{P_1\cap P_2}}(P_2,\ns)|$ and $|\hom_f(P,\ns)|$ are coprime with $|\ab|$, so the averages in the conclusion are well-defined. Let  $\varphi:\mb{Z}^r\to \ab$ be a surjective homomorphism, and let $s=(s_1,\ldots,s_r):\ab\to \mb{Z}^r$ be a cross-section. Then, since $s_j\co h$ is a $\mb{Z}$-valued function on $\hom_{f|_{P_1\cap P_2}}(P_2,\ns)$, by \cite[Lemma 2.2.14]{Cand:Notes2} we  have in $\mb{R}$ the following equality: $\mb{E}_{x\in \hom_{f|_{P_1\cap P_2}}(P_2,\ns)} s_j\co h(x) = \mb{E}_{x\in \hom_f(P,\ns)} s_j\co h\co \psi(x)$ for every $j\in [r]$. Using this in each of the $r$ components of $\mb{Z}^r$ we deduce that in $\mb{Z}^r$ we have the equality
\[
|\hom_f(P,\ns)|\sum_{x\in \hom_{f|_{P_1\cap P_2}}(P_2,\ns)} s\co h(x) = |\hom_{f|_{P_1\cap P_2}}(P_2,\ns)|\sum_{x\in \hom_f(P,\ns)} s\co h\co \psi(x).
\]
Applying $\varphi$ to both sides of this equality, the result follows.
\end{proof}
\noindent Lemma \ref{lem:coprime-hom} justifies the definition of the following operator on functions $f:\cu^{k+1}(\nss)\to\ab$:
\begin{equation}\label{eq:avop1}
\mc{E}(f)(\q_1\times\q_2):=  \mb{E}_{\q_1'\in \cu^{k+1}(\nss_1)} f(\q_1'\times\q_2).  
\end{equation}
\begin{lemma}\label{lem:kappa}
If $\rho:\cu^{k+1}(\nss_1\times\nss_2)\to\ab$ is a cocycle then $\mc{E}(\rho)$ is a cocycle which factors through $\pi_2^{\db{k+1}}$.
\end{lemma}
\begin{proof}
The fact that $\mc{E}(\rho)$ factors through $\pi_2^{\db{k+1}}$ follows clearly from \eqref{eq:avop1}, since for any $\q,\q'\in \cu^{k+1}(\nss_1\times\nss_2)$ satisfying $\q_2=\q_2'$, we have that $\mc{E}(\rho)(\q)$ and $\mc{E}(\rho)(\q')$ are averages of $\rho$ over the same set $\cu^{k+1}(\nss_1)\times\{\q_2\}$, so $\mc{E}(\rho)(\q) = \mc{E}(\rho)(\q')$.  

To see that $\mc{E}(\rho)$ is a cocycle, we check that it satisfies the two properties in \cite[Definition 3.3.14]{Cand:Notes1}. The first property follows from the additivity of the averaging operator $\mc{E}$ and the fact that $\rho$ itself satisfies the property. To check the second property (additivity relative to concatenations), let $\q=\q_1\times\q_2$ be adjacent to $\q''=\q_1''\times\q_2''$, with concatenation $\tilde\q=\tilde\q_1\times\tilde\q_2$ (in fact only second-component cubes matter here). Recalling the notation $q\prec q'$ for adjacency of cubes $q,q'$ (see \cite[Definition 3.1.6]{Cand:Notes1}), let
\begin{equation}\label{eq:Qset1}
Q:=\{(q_1,q_1')\in \cu^{k+1}(\nss_1)^2:  q_1\prec q_1'\}.
\end{equation}
Note that $Q$ is the morphism set $\hom(P,\nss_1)$ where $P$ is the union of two adjacent $(k+1)$-faces $F_1=\{v:v\sbr{1}=0\}$, $F_2=\{v:v\sbr{2}=0\}$ in $\db{k+2}$. Let $D:=\{v:v\sbr{1}=1-v\sbr{2}\}=P\setminus (F_1\cap F_2)$. As $\emptyset$ and $D$ form a good pair,\footnote{Indeed, for any abelian group $\ab$ and $f\in \hom(D,\mc{D}_k(\ab))$, we can extend $f$ to $\tilde{f}\in \hom(P,\mc{D}_k(\ab))$  by setting $\tilde{f}(0,0,v\sbr{3},\ldots,v\sbr{k+1}):=f(1,0,v\sbr{3},\ldots,v\sbr{k+1})$.} by Corollary \ref{cor:mes-presev-with-torsion}  we have that $\mc{E}(\rho)(\tilde \q) \; = \mb{E}_{f=(q_1,q_1')\in Q}\; \rho(f|_D\times \tilde\q_2)$ (note that this follows also from the idempotence property for nilspaces, see \cite[Proposition 3.6]{CScouplings}). Using that $f|_D =q_1\prec q_1'$ and \cite[Definition 3.3.4 (ii)]{Cand:Notes1} we get that $\mc{E}(\rho)(\tilde \q) \; = \mb{E}_{f=(q_1,q_1')\in Q}\; \rho(f|_D\times \tilde\q_2) = \mb{E}_{f=(q_1,q_1')\in Q}\; \rho(q_1\times \q_2) +\rho(q_1'\times \q_2'')$. Moreover, for $i\in[2]$ we have that $\emptyset$ and $F_i$ also form a good pair. Using Corollary \ref{cor:mes-presev-with-torsion} again, we have 
$\mc{E}(\rho)(\tilde \q) =\mb{E}_{\q_1'\in \cu^{k+1}(\nss_1)} \; \rho(\q_1'\times \q_2)+\mb{E}_{\q_1'\in \cu^{k+1}(\nss_1)} \; \rho(\q_1'\times \q_2')=\mc{E}(\rho)(\q)+ \mc{E}(\rho)(\q'')$, 
which completes the proof.
\end{proof}

\noindent The second operator that we shall use is the following, which averages cocycles over a smaller set than $\mc{E}$, namely over essentially a \emph{rooted} cube set $\cu^n_y(\nss_1):=\{\q\in\cu^n(\nss_1):\q(0^n)=y\}$. For any $f:\cu^{k+1}(\nss)\to\ab$ we define (using again Lemma \ref{lem:coprime-hom})
\begin{equation}\label{eq:avop2}
\mc{E}'(f)(\q):= \mb{E}_{\q_1'\in \cu^{k+1}_{\q_1(0^{k+1})}(\nss_1)} f(\q_1'\times\q_2)\; \textrm{ for any $\q=\q_1\times\q_2\in\cu^{k+1}(\nss)$.}  
\end{equation}
Note that, unlike for $\mc{E}$, here for a cocycle $\rho$ the map $\mc{E}'(\rho)$ is not necessarily a cocycle; indeed the additivity relative to concatenations can fail. However, the difference $\mc{E}(\rho)-\mc{E}'(\rho)$ has the following property, which will be another central ingredient in the proof of Proposition \ref{prop:cocycledecomp}.
\begin{lemma}\label{lem:fact1}
If $\q,\q'\in\cu^{k+1}(\nss)$ are adjacent along an upper $k$-face $F$ of $\db{k+1}$, 
then, letting $\q''$ be the concatenation of $\q,\q'$ along this face \textup{(}where $\q\prec_F\q'$\textup{)}, we have\footnote{A $k$-upper face $F\subset \db{k+1}$ is a set of the form $\{v\in \db{k+1}:v\sbr{i}=1\}$ for some $i\in[k]$. For  $\q,\q'\in \cu^{k+1}(\nss)$ we write $\q\prec_F \q'$ if $\q(v)= \q'(\sigma_i(v))$ for every $v\in F$, where $\sigma_i(v)$ switches the coordinate $v\sbr{i}$ to $1-v\sbr{i}$. The concatenation $\q''$ of $\q,\q'$ along $F$ is then defined by $\q''(v):=\q(v)$ when $v\sbr{i}=0$ and $\q''(v):=\q'(v)$ for $v\in F$. This generalizes the usual adjacency $\q\prec \q'$, which is the special case of $\q\prec_F \q'$ with $F=\{v:v\sbr{k+1}=1\}$.}
\begin{equation}
[\mc{E}(\rho)-\mc{E}'(\rho)](\q'')=[\mc{E}(\rho)-\mc{E}'(\rho)](\q).    
\end{equation}
\end{lemma}
\begin{proof}
Note that $\mc{E}'(\rho)$ satisfies the first property in \cite[Definition 3.3.14]{Cand:Notes1} for automorphisms that fix $0^{k+1}$, so we can suppose without loss of generality that the upper face in question is $\{v\in\db{k+1}:v\sbr{k+1}=1\}$ (as in the usual definition of adjacency and concatenations \cite[Definition 3.1.16]{Cand:Notes1}). Let $\q=\q_1\times\q_2$ and similarly $\q'=\q_1'\times\q_2'$, $\q''=\q_1''\times\q_2''$.

We then define the following \emph{rooted} variant of \eqref{eq:Qset1}:
\begin{equation}\label{eq:Qset2}
Q'(\q_1):=\{(q_1,q_1')\in\cu^{k+1}(\nss_1)^2: q_1\prec q_1'\textrm{ and }q_1(0^{k+1})=\q_1(0^{k+1})\}.
\end{equation}
Since the cocycle $\mc{E}(\rho)$ satisfies $\mc{E}(\rho)(\q'')=\mc{E}(\rho)(\q)+\mc{E}(\rho)(\q')$, it suffices to prove that
\begin{equation}
    \mc{E}'(\rho)(\q'') = \mc{E}'(\rho)(\q)+\mc{E}(\rho)(\q').
\end{equation}
For every $(q_1,q_1')\in Q'(\q_1)$, their concatenation $q_1''$ satisfies $\rho(q_1''\times \q_2'')=\rho(q_1\times \q_2)+\rho(q_1'\times \q_2')$. Averaging this over $Q'(\q_1)$ (using Lemma \ref{lem:coprime-hom}) we have
\[
\mb{E}_{Q'(\q_1)}\rho(q_1''\times \q_2'')=\mb{E}_{Q'(\q_1)}\rho(q_1\times \q_2)+\mb{E}_{Q'(\q_1)}\rho(q_1'\times \q_2').
\]
Note that the left side here is $\mc{E}'(\rho)(\q'') $ because we can apply\footnote{Indeed, letting $F_1=\{v\in\db{k+2}:v\sbr{2}=0\}$, $F_2:=\{v\in\db{k+2}:v\sbr{1}=0\}$, $P:=F_1\cup F_2$ and $v=(1,0^{k+1})$, note that  $P\setminus(F_1\cap F_2)$ has the extension property in $P$, so $\{v\}$, $P\setminus(F_1\cap F_2)$ form a good pair.}  Corollary \ref{cor:mes-presev-with-torsion} to the map $Q'(\q_1)\to \cu_{\q_1(0^{k+1})}^{k+1}(\nss_1)$,  $(q_1,q_1')\mapsto q_1''$. Similarly, the first summand on the right side above is $\mc{E}'(\rho)(\q)$ because we can also apply\footnote{Using the same definitions as in the previous footnote, we have to prove that $\{v\}$, $F_1$ form a good pair. This follows by the extension property of $F_1$.} Corollary \ref{cor:mes-presev-with-torsion} to the map $Q'(\q_1)\to \cu_{\q_1(0^{k+1})}^{k+1}(\nss_1)$,  $(q_1,q_1')\mapsto q_1$. Finally, the second summand on the right side is $\mc{E}(\rho)(\q')$ again by Corollary \ref{cor:mes-presev-with-torsion} applied\footnote{For this last part we need a slightly different argument. Here we need to prove that $\{v\}$, $F_2$ form a good pair. Note that $F_2\cup \{v\}$ is a simplicial cubespace as per \cite[Definition 3.1.4]{Cand:Notes1}, so the result follows by \cite[Lemma 3.1.5]{Cand:Notes1}.}  to $Q'(\q_1)\to \cu^{k+1}(\nss_1)$,  $(q_1,q_1')\mapsto q_1'$.
\end{proof}
\begin{proof}[Proof of Proposition \ref{prop:cocycledecomp}]
We set $\kappa=\mc{E}(\rho)$, which by Lemma \ref{lem:kappa} is a cocycle that factors through $\pi_2^{\db{k+1}}$. Thus our main task is to prove that $\rho-\mc{E}(\rho)$ is a coboundary, i.e., that there is a function $g:\nss \to \ab$ such that the following holds for every cube $\q\in\cu^{k+1}(\nss)$:
\begin{equation}\label{eq:cobfact}
\rho(\q)-\mc{E}(\rho)(\q)=\sigma_{k+1}(g\co\q).
\end{equation}
We can describe the function $g$ explicitly: for each $y \in \nss$ we let
\[
g(y):= \mc{E}'(\rho)(\q)-\mc{E}(\rho)(\q)\;\; \textrm{for any $\q\in \cu^{k+1}_y(\nss)$.}
\]
Let us prove that $g$ is thus a well-defined map. Equivalently, it suffices to show that the function 
$\wt{g}:\cu^{k+1}(\nss)\to\ab,\; \q\mapsto \mc{E}'(\rho)(\q)-\mc{E}(\rho)(\q)$ depends only on the value  $\q(0^{k+1})$, i.e., \begin{equation}\label{eq:rootdep}
\textrm{if $\q,\q'\in\cu^{k+1}(\nss)$ satisfy $\q(0^{k+1})=\q'(0^{k+1})$ then $\wt{g}(\q)=\wt{g}(\q')$.}    
\end{equation}
To see this, fix any $\q,\q'\in \cu^{k+1}(\nss)$ such that $\q(0^{k+1})=\q'(0^{k+1})$. The results in \cite[\S 2.2.3]{Cand:Notes2} imply that the map sending each tricube morphism $t\in \hom(T_{k+1},\nss)$ to the pair of cubes\footnote{Here $\omega_{k+1}$ is the outer-point map, see \cite[Definition 3.1.15]{Cand:Notes1}.} $(t\co\Psi_{0^{k+1}},t\co\omega_{k+1})$ is Haar-measure preserving, whence surjective (by finiteness of $\nss$), so there exists  $t\in \hom(T_{k+1},\nss)$ such that $t\co\Psi_{0^{k+1}}=\q$ and such that the outer cube $t\co \omega_{k+1}$ equals $\q'$. Starting from $t\co\Psi_{0^{k+1}}=\q$, using a straightforward sequence of concatenations along internal $k$-faces of $t$ we deduce using Lemma \ref{lem:fact1} that $\wt{g}(\q)=\wt{g}(t\co \omega_{k+1})=\wt{g}(\q')$. This proves \eqref{eq:rootdep}.

To complete the proof we now show that the function $g$ satisfies \eqref{eq:cobfact}. For this we use an argument consisting in averaging cocycles over tricubes, inspired from similar arguments used in previous work (e.g.\ in \cite[\S 4]{CSinverse}). For any fixed cube $\q=\q_1\times\q_2\in\cu^{k+1}(\nss)$, note that for any tricube  $t_i$ on $\nss_i$ with outer-cube $\q_i$ ($i=1,2$) we have $\rho(\q)=\sum_{v\in \db{k+1}}(-1)^{|v|}\rho((t_1\times t_2)\co\Psi_v)$. Recall that by \cite[Lemma 3.1.17]{Cand:Notes1} the set of tricubes $T_{k+1}$ can be identified with a simplicial set inside $\db{2(k+1)}$ via an injective morphism $q:T_{k+1}\to \db{2(k+1)}$. Letting $\mc{T}_{\q_1}$ denote the set of all tricubes on $\nss_1$ with outer-cube $\q_1$, note that $|\mc{T}_{\q_1}|$ and $|\ab|$ are coprime by Lemma \ref{lem:coprime-hom} (using that the outer point set of a tricube has the extension property by \cite[Lemma 2.2.21]{Cand:Notes2}) and we therefore have
\begin{eqnarray*}
\rho(\q)& =& \mb{E}_{t_1\in \mc{T}_{\q_1}}\sum_{v\in \db{k+1}}(-1)^{|v|}\rho\big((t_1\times t_2)\co\Psi_v\big) = \sum_{v\in \db{k+1}}(-1)^{|v|} \mb{E}_{t_1\in \mc{T}_{\q_1}}\rho((t_1\co\Psi_v)\times (t_2\co\Psi_v))\\
& =&\sum_{v\in \db{k+1}}(-1)^{|v|} \mb{E}_{\q_1'\in \cu^{k+1}_{\q_1(v)}(\nss_1)}\;\rho(\q_1' \times (t_2\co\Psi_v)),
\end{eqnarray*}
where the last equality follows from the fact that for each fixed $v$ we can apply Corollary \ref{cor:mes-presev-with-torsion} to the map $t\mapsto t\co\Psi_v$ (using again \cite[Lemma 2.2.21]{Cand:Notes2}). On the other hand, since $\mc{E}(\rho)$ is a cocycle, we have $\mc{E}(\rho)(\q) = \sum_{v\in \db{k+1}}(-1)^{|v|}\mc{E}(\rho)\big((t_1\times t_2)\co\Psi_v\big)$, and then averaging again both sides over $\mc{T}_{\q_1}$, we obtain $\mc{E}(\rho)(\q) = \sum_{v\in \db{k+1}}(-1)^{|v|}\mb{E}_{\q_1'\in \cu^{k+1}_{\q_1(v)}(\nss_1)}\mc{E}(\rho)\big(\q_1'\times (t_2\co\Psi_v)\big)$. Subtracting the last two equalities we confirm \eqref{eq:cobfact} as follows:
\begin{eqnarray*}
\rho(\q)-\mc{E}(\rho)(\q) & = & \sum_{v\in \db{k+1}}(-1)^{|v|} \; \mb{E}_{\q_1'\in \cu^{k+1}_{\q_1(v)}(\nss_1)}\;\Big(\rho(\q_1' \times (t_2\co\Psi_v)) - \mc{E}(\rho)(\q_1'\times (t_2\co\Psi_v))\Big)\\
& = & \sum_{v\in \db{k+1}}(-1)^{|v|} \; g(\q_1(v),\q_2(v)) = \sigma_{k+1}(g\co\q).
\end{eqnarray*}
This completes the proof.
\end{proof}
\noindent We have thus completed the proof of Theorem \ref{thm:p-sep}. Before we start deriving its consequences, let us prove the uniqueness of the primary decomposition, claimed in Theorem \ref{thm:p-sylow-intro}. This relies on the following result.

\begin{theorem}\label{thm:unique-p-sylow}
Let $\Lambda$ be a finite set of prime numbers, for each $p\in\Lambda$ let $\ns_p,\nss_p$ be $p$-nilspaces, and suppose that $\varphi: \prod_{p \in \Lambda}\ns_p \to \prod_{p \in \Lambda}\nss_p $ is a nilspace isomorphism. Then there exist nilspace isomorphisms $\phi_p:\ns_p\to \nss_p$ such that $\varphi(x)=(\phi_p(x_p))_{p\in \Lambda}$ for every $x=(x_p)_{p\in \Lambda}\in \prod_{p \in \Lambda}\ns_p$.
\end{theorem}

We split the proof into the following lemmas.

\begin{lemma}\label{lem_coprime-implies-cte-1}
Let $\ns$ be a $k$-step finite nilspace, let $\ab$ be a finite abelian group such that $|\ns|$, $|\ab|$ are coprime, and let $\ell\in \mb{N}$. Then  $\hom(\ns,\mc{D}_\ell(\ab))$ is the set of constant functions $\ns\to\ab$.
\end{lemma}

\begin{proof}
Clearly there exists some integer $m$ coprime with $|\ab|$ such that all structure groups of $\ns$ have torsion $m$. Let $\varphi\in \hom(\ns,\mc{D}_\ell(\ab))$. By Corollary \ref{cor:fin-ns-quot-modN} we can assume that $\ns=\prod_{i=1}^k \mc{D}_i(\mb{Z}_{m^\gamma}^{a_i})$ for some $\gamma\in \mb{N}$ and $a_i\ge 0$ for $i\in[k]$. Note that for any fixed $(x_2,\ldots,x_k)\in \prod_{i=2}^k \mc{D}_i(\mb{Z}_{m^\gamma}^{a_i})$, the map $\mc{D}_1(\mb{Z}_{m^\gamma}^{a_1})\to \mc{D}_\ell(\ab)$,  $x\mapsto \varphi(x,x_2,\ldots,x_k)$ is a morphism. By \cite[Lemma 6.6]{CSinverse} this map is constant. Hence $\varphi$ depends only on the last $k-1$ components, i.e., for any fixed $\tilde{x_1}\in \mc{D}_1(\mb{Z}_{m^\gamma})^{a_1}$, for every $(x_1,x_2,\ldots,x_k)\in \prod_{i=1}^k \mc{D}_i(\mb{Z}_{m^\gamma})^{a_i}$ we have $\varphi(x_1,\ldots,x_k)=\varphi(\tilde{x_1},x_2,\ldots,x_k)$. Now we can repeat the same argument, fixing the last $k-2$ coordinates we define a map $\mc{D}_2(\mb{Z}_{m^\gamma}^{a_2})\to \mc{D}_\ell(\ab)$ as $x\mapsto \varphi(\tilde{x_1},x,x_3,\ldots,x_k)$. By \cite[Lemma A.2]{CGSS-doublecoset} this map can be viewed as a morphism $\mc{D}_1(\mb{Z}_{m^\gamma}^{a_2})\to \mc{D}_{\lfloor\ell/2\rfloor}(\ab)$ and hence, again by \cite[Lemma 6.6]{CSinverse}, this map is constant. Thus, for any fixed $\tilde{x_2}\in \mc{D}_2(\mb{Z}_{m^\gamma})^{a_2}$ we have $\varphi(\tilde{x_1},\tilde{x_2},x_3,\ldots,x_k)=\varphi(\tilde{x_1},x_2,\ldots,x_k)=\varphi(x_1,x_2,\ldots,x_k)$ for all $(x_1,x_2,\ldots,x_k)\in \prod_{i=1}^k \mc{D}_i(\mb{Z}_{m^\gamma})^{a_i}$. Therefore $\varphi$ does not depend on the first 2 coordinates. Continuing this way we  eventually conclude that $\varphi$ is constant.
\end{proof}

\begin{lemma}\label{lem:coprime-implies-cte-2}
Let $\ns,\nss$ be finite nilspaces such that $|\ns|$ and $|\nss|$ are coprime. Then $\hom(\ns,\nss)$ is the set of constant functions $\ns\to\nss$.
\end{lemma}

\begin{proof}
Suppose that $\ns$ is $k$-step and $\nss$ is $\ell$-step. We argue by induction on $\ell$. For $\ell=1$ the result follows from Lemma \ref{lem_coprime-implies-cte-1}. For $\ell>1$, given any $\varphi\in\hom(\ns,\nss)$, note that $\pi_{\ell-1}\co \varphi \in \hom(\ns,\nss_{\ell-1})$, so by induction this map is constant. Thus $\varphi$ takes values in a single fiber of $\pi_{\ell-1}$, so we can view $\varphi$ as a morphism $\ns\to \mc{D}_k(\ab_\ell(\nss))$, which is constant by Lemma \ref{lem_coprime-implies-cte-1}.
\end{proof}

\begin{remark}
Note that Lemma \ref{lem:coprime-implies-cte-2} generalizes both \cite[Lemma 6.7]{CSinverse} and \cite[Lemma 8.1]{JST-tot-dis}.
\end{remark}

\noindent Theorem \ref{thm:unique-p-sylow} now follows from a simple application of Lemma \ref{lem:coprime-implies-cte-2}. In fact we use a similarly simple application to give a more general result in the next subsection; see Proposition \ref{prop:pmorphismsep}.

\begin{remark}\label{rem:JSTimplication}
We did not find a simple deduction of Theorem \ref{thm:p-sylow-intro} from \cite[Theorem 2.3]{JST-tot-dis}. A natural attempt could consist in using the action of the translation group $\tran(\ns)$ on the given finite nilspace $\ns$. However, this is in general not an \emph{abelian} group $\Gamma$ as in \cite[Theorem 2.3]{JST-tot-dis}. There are also issues involving the possible  lack of ergodicity (or transitivity, in this finite setting) of this action, which are not easily dealt with, even invoking the ergodic decomposition.
\end{remark}

\vspace{1cm}

\subsection{Some consequences of the primary decomposition of finite nilspaces}\label{Subsec:PrimDecompConseq}\hfill\smallskip\\
We begin with Theorem \ref{thm:invreduc-intro}, of which we shall in fact prove the following refinement.
\begin{theorem}\label{thm:invreduc}
Let $k,m\in \mb{N}$ and $\delta > 0$. Let $\ab$ be a finite $m$-torsion abelian group with primary decomposition $\ab=\bigoplus_{p\in \mc{P}} \ab_p$, for $\mc{P}=\mc{P}(|\ab|)$, and $f: \ab\to \mb{C}$ be 1-bounded. Suppose that there is an $m$-torsion nilspace polynomial $F\co\phi$ of degree $k$ and complexity $M$ such that $|\mb{E}_{z\in \ab} f(z) F\co \phi(z) |\geq \delta$. Then for some $\varepsilon=\varepsilon(\delta,M)>0$, for each $p\in\mc{P}$ there exists a $p$-nilspace polynomial $F_p\co \phi_p:\ab_p\to\mb{C}$ of complexity at most $M'(M)$ such that $|\mb{E}_{z\in \ab} f(z)  \prod_{p\in \mc{P}} F_p\co \phi_p(z_p) |\geq \varepsilon$. Moreover, for any $b>0$ there exists $b'=b'(M,b)>0$ such that if $\phi$ is $b'$-balanced, then for every $p\in\mc{P}$ the morphism $\phi_p:\mc{D}_1(\ab_p)\to \ns_p$ is $b$-balanced.
\end{theorem}
\noindent Note that Theorem \ref{thm:invreduc-intro} follows readily by combining Theorem \ref{thm:invreduc} with Theorem \ref{thm:nil-for-m-exp}. To prove Theorem \ref{thm:invreduc} we will use the following result, extending Lemma \ref{lem:coprime-implies-cte-2}.

\begin{proposition}\label{prop:pmorphismsep}
Let $\mc{P}$ be a finite set of primes, let $\ns,\nss$ be nilspaces of the form $\ns=\prod_{p\in \mc{P}}\ns_p$, $\nss=\prod_{p\in \mc{P}}\nss_p$, where $\ns_p$, $\nss_p$ are $k$-step finite $p$-nilspaces for each $p\in \mc{P}$. Let $\phi:\ns\to\nss$ be a morphism. Then for each $p\in \mc{P}$ there exists a morphism $\phi_p:\ns_p\to \nss_p$ such that for every $x=(x_p)_{p\in\mc{P}}$ we have $\varphi(x)=\big(\phi_p(x_p)\big)_{p\in \mc{P}}$.
\end{proposition}
\begin{proof}
For each prime $p\in\mc{P}$, for each $x_p\in\ab_p$, the map $x_p'\to \pi_p\co \phi(x_p,x_p')$ is a morphism $\ns_p'\to\ns_p$ (where $\ns_p'$ denotes the complement $\prod_{p'\neq p} \ns_{p'}$ and each $x\in \ab$ can be written $(x_p,x_p')$ with $x_p\in\ns_p$ and $x_p'\in\ns_p'$) and must therefore be constant, by Lemma \ref{lem:coprime-implies-cte-2}. Therefore there is a morphism $\phi_p:\ns_p \to\nss_p$ such that $\pi_p\co\phi(x)=\phi_p(x_p)$. We thus obtain that
$\phi(x)=\big(\pi_p(\phi(x))\big)_{p\in\mc{P}}=\big(\phi_p(x_p)\big)_{p\in\mc{P}}$, as claimed.
\end{proof}

\begin{proof}[Proof of Theorem \ref{thm:invreduc}]
Let $\ns$ be the $m$-torsion nilspace such that $\phi:\ab\to\ns$, $F:\ns\to\mb{C}$, and $|\langle f,F\co\phi\rangle|\geq\delta$. Let $X\cong\prod_{p\in\mc{P}} \ns_p$ be the decomposition given by Theorem \ref{thm:p-sep}, where $\mc{P}=\mc{P}(|\ns|)$. Applying Proposition \ref{prop:pmorphismsep}, we deduce that without loss of generality $\mc{P}$ is also the set of primes dividing $|\ab|$, so by the primary decomposition for finite abelian groups we have $\ab\cong\prod_{p\in\mc{P}} \ab_p$, and moreover for each $p\in \mc{P}$ there is a morphism $\phi_p:\ab_p\to \ns_p$ such that for every $x=(x_p)_{p\in\mc{P}}\in\ab$ we have $\phi(x)=(\phi_p(x_p))_{p\in\mc{P}}$. 

Now note that we can write $F$ as a sum of 1-bounded rank-1 functions $\prod_{p\in \mc{P}} F_{p,i}(x_p)$, $i\in [R]$ with $R=R(M)$. We can do this very crudely as follows: for $x=(x_p)_{p\in \mc{P}}\in \ns$ we simply write $F(x)=\sum_{(y_p)_{p\in \mc{P}}\in \ns} F((y_p)_{p\in \mc{P}})\prod_{(y_p)_{p\in \mc{P}}}1_{y_p}(x_p)$ where $1_{y_p}(x_p)=1$ if $x_p=y_p$ and $0$ otherwise. (Note that then $R=|\ns|$ which is bounded in terms of $M$.)

Given the above rank-1 decomposition of $F$, we now apply the pigeonhole principle to deduce that there is $i\in [R]$ such that 
$|\langle f,\prod_{p\in \mc{P}} F_{p,i}\co\phi_p\rangle| \geq \delta/R$, so we can let $\varepsilon=\delta/R$.

To deduce the second part of the theorem, note that since $\ns\cong \prod_{p\in \mc{P}} \ns_p$, we have that $\pr_p:\ns\to \ns_p$ is a fibration. In particular, for any $n\ge 0$ it induces a continuous map $\mc{P}(\cu^n(\ns))\to \mc{P}(\cu^n(\ns_p))$ (here $\mc{P}(\cu^n(\ns))$ and $\mc{P}(\cu^n(\ns_p))$ are the spaces of probability measures on $\cu^n(\ns)$ and $\cu^n(\ns_p)$ respectively) defined as $\mu\mapsto \mu\co \pr_p^{-1}$ (see the proof of \cite[Theorem 1.3]{CGSS-p-hom}). Thus, for each $p\in\mc{P}$ we can assume that $\pr_p\co \phi$ is $b$-balanced for $b'=b'(\ns)>0$ small enough.
\end{proof}

\begin{corollary}\label{cor:Q1.9reduc}
A positive answer to Question \ref{mainQ:p-case} implies a positive answer to Question \ref{Q:JST}.
\end{corollary}
\begin{proof}
This follows by a straightforward adaptation of the proof of \cite[Lemma 6.5]{CGSS-p-hom}. More precisely, starting from the conclusion of Theorem \ref{thm:invreduc}, for each $p\in\mc{P}:=\mc{P}(|\ab|)$ we have a nilspace polynomial $F_p\co \phi_p$ such that $|\mb{E}_{z\in \ab} f(z)  \prod_{p\in \mc{P}} F_p\co \phi_p(z_p) |\geq \varepsilon>0$. Our goal is to replace the product here by a polynomial phase function. Note that we can assume that $\phi_p:\mc{D}_1(\ab_p)\to \ns_p$ is arbitrarily well-balanced for each $p$. Given the positive answer to Question \ref{mainQ:p-case}, for each $p\in \mc{P}$ there exists an abelian  $p$-group nilspace\footnote{Thus $\nss_p$ is an abelian $p$-group equipped with the Host--Kra cubes for a certain degree-$k$ filtration on this group.}  $\nss_p$, a morphism $g_p:\mc{D}_1(\ab_p)\to \nss_p$, and a map $h_p:\nss_p\to \ns_p$, such that $\phi_p=h_p\co g_p$. By Fourier analysis in the finite abelian group underlying $\nss_p$, we can write the function $F_p\co h_p:\nss_p\to \mb{C}$ as $F_p\co h_p(y_p)=\sum_{\chi^{(p)}}\lambda_{\chi^{(p)}} \chi^{(p)}(y_p)$, where the sum is over the Pontryagin dual of the group underlying $\nss_p$, dual group of order $O_{\delta,m,k}(1)$. Thus $F_p\co \phi_p(z_p) = \sum_{\chi^{(p)}}\lambda_{\chi^{(p)}} \chi^{(p)}(g_p(z_p))$.

It follows that $\varepsilon<|\langle f,\prod_{p\in\mc{P}} \sum_{\chi^{(p)}\in \widehat{\nss_p}}\lambda_{\chi^{(p)}} \chi^{(p)}(g_p(z_p))\rangle|$, where, expanding the product in the right side here, we see that this equals $|\langle f, \sum_{(\chi^{(p)}\in \widehat{\nss_p})_{p\in\mc{P}}}\prod_{p\in\mc{P}}\lambda_{\chi^{(p)}} \chi^{(p)}(g_p(z_p))\rangle|$. Note that each product in the latter expression is a scalar times a polynomial phase function. Since for every $p\in\mc{P}$ we have $|\nss_p|=O_{\delta,m,k}(1)$, the result follows by the pigeonhole principle.
\end{proof}

\noindent As mentioned in the introduction, we can now easily establish a positive answer for Question \ref{Q:JST} in the squarefree case.
\begin{corollary}\label{cor:inv-disct-primes}
Let $p_1,p_2,\ldots,p_r$ be distinct primes, let $m=p_1\cdots p_r$, let $\ab$ be a finite abelian group of torsion $m$, and let $\ab\cong \bigoplus_{i\in [r]} \ab_i$ be the primary decomposition of  $\ab$, where $\ab_i$ is a $p_i$-group for each $i\in [r]$. Then for any $k\in \mb{N}$ and $\delta>0$ there exists $\varepsilon=\varepsilon(\delta,k,m)>0$ such that the following holds.  For any 1-bounded function $f:\ab\to \mb{C}$ with $\|f\|_{U^{k+1}}>\delta$ there exists for each $i\in [r]$ a degree-$k$ polynomial $\phi_i\in\hom\big(\mc{D}_1(\ab_i),\mc{D}_k(\mb{T})\big)$ such that
$|\mb{E}_{x\in\ab} f(x)\overline{e(\phi_1(x_1)+\cdots+\phi_r(x_r))}|>\varepsilon$.
\end{corollary}

\begin{proof}
This follows from (the proof of) Corollary \ref{cor:Q1.9reduc} combined with the fact, proved in \cite{CGSS-p-hom}, that the answer to Question \ref{mainQ:p-case} is positive when $\ab$ is an \emph{elementary} abelian $p$-group. Let us give more detail. 
By \cite[Theorem 1.3]{CGSS-p-hom}, for any $k$-step \textsc{cfr} nilspace $\nss$ and any prime $p$ there exists $b=b(\nss,p)>0$ such that if there exists a $b$-balanced morphism $\phi:\mc{D}_1(\mb{Z}_p^n)\to \nss$ then $\nss$ is $p$-homogeneous.\footnote{This constant depends also on the particular metric chosen on the space of measures on $\cu^n(\nss)$, but we do not make this explicit here as we assume that each $k$-step \textsc{cfr} nilspace has one such metric fixed.} Combining Theorems \ref{thm:nil-for-m-exp} and \ref{thm:invreduc} we deduce that $f$ correlates with a product $\prod_{i\in [r]}F_i\co\phi_i$ where $\phi_i:\mc{D}_1(\ab_i)\to \ns_i$ are nilspace morphisms such that the size of the $p_i$-nilspace $\ns_i$ is bounded by a constant $M=M(\delta,m,k)>0$, for every $i\in [r]$. Moreover, by the second part of both theorems, we can assume that if the morphism $\phi:\mc{D}_1(\ab)\to \ns\cong\prod_{i\in [r]} \ns_i$ defined as  $z=(z_i)_{i\in [r]}\in \ab\mapsto \phi_i(z_i)$ is sufficiently balanced, then each $\phi_i$ is also $b$-balanced for any $b=b(\ns_i)$. In particular, we can choose this $b$ to be smaller than $\min_{i\in [r]} b(\ns_i,p_i)>0$ where $b(\ns_i,p_i)$ is the constant described in the previous paragraph. Therefore, we can assume that each $\ns_i$ is a $k$-step, $p_i$-homogeneous nilspace (by \cite[Theorem 1.3]{CGSS-p-hom}).

Now fix any $i\in [r]$ and consider $\phi_i:\mc{D}_1(\ab_i)\cong \mc{D}_1(\mb{Z}_{p_i}^n)\to \ns_i$. Then \cite[Theorem 1.7]{CGSS-p-hom}  answers Question \ref{mainQ:p-case} positively in the present case, namely, it provides a finite $k$-step, $p_i$-homogeneous \emph{abelian group} nilspace $\nss_i$ with size bounded in terms of the size of $\ns_i$ (and thus bounded in terms of $\delta,m$ and $k$), a fibration $h_i:\nss_i\to \ns_i$ and a morphism $g_i:\mc{D}_1(\ab_i)\to\nss_i$ such that $h_i\co g_i = \phi_i$. The rest of the argument can then be carried out using standard Fourier analysis on the abelian group $\nss_i$, as in the proof of Corollary \ref{cor:Q1.9reduc}.
\end{proof}

\begin{remark}
We included Corollary \ref{cor:inv-disct-primes} as a relatively simple illustration of the usefulness of the reduction to the $p$-group case afforded by Theorem \ref{thm:invreduc-intro}. After the first preprint version of this paper appeared, another way to obtain Corollary \ref{cor:inv-disct-primes} was observed by Jamneshan, Shalom and Tao (in personal communication). In a nutshell, this alternative proof argues by contradiction, taking the ultraproduct of supposed counterexamples to the desired inverse theorem and equipping this with Loeb measure to obtain an ergodic system, then using the ergodic theoretic Sylow decomposition \cite[Theorem 2.3]{JST-tot-dis} to separate this into ergodic factors that are then shown to be based on $p$-homogeneous nilspaces, and then combining this with our results from \cite{CGSS-p-hom} describing such nilspaces, in order to conclude. Due to the additional technicalities and background, we do not detail further this alternative argument here.
\end{remark}

\noindent We now establish the relation between the uniformity norms and the $(n,d)$-cut norms.

\begin{proof}[Proof of Corollary \ref{cor:box-norm-estimate}]
First, for each $i\in[n]$, decompose  $m_i=\prod_{j=1}^{r_i}p_{i,j}^{a_{i,j}}$ for some integers $r_i\ge 1$, $a_{i,j}\ge 1$, and primes $p_{i,j}$ for $j\in[r_i]$. Note that as the $\{m_i\}_{i\in[r]}$ are pairwise coprime, the primes $\{p_{i,j}\}_{i\in[n],j\in[r_i]}$ are all distinct. Thus we can apply Theorem \ref{thm:invreduc-intro} with $m:=\prod_{i=1}^n m_i = \prod_{i=1}^n\prod_{j=1}^{r_i}p_{i,j}^{a_{i,j}}$ and deduce that $f$ correlates with a function of the form $\prod_{p\mid m} F_p\co \phi_p(z_p)$ where the $F_p$ are 1-bounded.  This clearly equals a function of the form $\prod_{i=1}^n u_i(z_i)$ where for $i\in[n]$,  $u_i:=\prod_{j\in[r_i]}F_{p_{i,j}}\co \phi_{p_{i,j}}$. By definition of the $(n,1)$-cut norm the result follows.
\end{proof}

\begin{remark}\label{rem:upformulation} Our results on primary decompositions of finite nilspaces have a surprising analytic consequence for ultraproduct groups. Assume that $\{A_i\}_{i=1}^\infty$ and $\{B_i\}_{i=1}^\infty$ are infinite sequences of finite abelian groups with uniformly bounded torsions such that $|A_i|$ and $|B_i|$ are coprime. Let $A$ (resp.\ $B$) be the ultra product of $\{A_i\}_{i=1}^\infty$ (resp.\ $\{B_i\}_{i=1}^\infty$) with respect to a fixed non-principal ultrafilter. Our results imply that for every $k$ the so called $k$-th Fourier $\sigma$-algebra on $A\times B$ (formed by $U^{k+1}$ structured sets; see \cite[Definition 3.18]{CScouplings}) is the product of the $k$-th Fourier $\sigma$-algebras of $A$ and $B$. Note that the assumptions of uniformly bounded torsions and coprime orders are both necessary for this to hold. Furthermore, the statement generalizes to more than two components $A,B$, assuming pairwise coprimality (as in Corollary  \ref{cor:box-norm-estimate}). 
\end{remark}

\noindent We close this section with the observation that Theorem \ref{thm:p-sep} can be extended from finite nilspaces to $m$-torsion \emph{compact} (not necessarily finite) nilspaces, as follows.
\begin{corollary}\label{cor:p-sylow-cpct}
Let $m\in \mb{N}$ and let $\ns$ be a compact $k$-step $m$-torsion nilspace. Then $\ns$ is isomorphic (as a \emph{compact} nilspace) to $\prod_{p| m}\ns_p$, where for each $p|m$ the nilspace $\ns_p$ is a compact $k$-step $p$-nilspace.\footnote{Note that for some $p|m$, the nilspace $\ns_p$ may be the trivial one point nilspace.} Moreover, this decomposition is unique up to isomorphisms.
\end{corollary}

\noindent We outline the proof as follows. Recall firstly that every compact $k$-step nilspace is the inverse limit of compact $k$-step nilspaces of \emph{finite rank} (see  \cite[Theorem 2.7.3]{Cand:Notes2}) and if $\ns$ has its structure groups of torsion $m$, so do all its factors in the inverse system. Thus, the factors in the inverse system are finite nilspaces, which by Theorem \ref{thm:p-sep} split into their $p$-Sylow components. Moreover, by Proposition \ref{prop:pmorphismsep} the morphisms in the inverse system also split according to these $p$-Sylow components. We conclude that in fact $\ns$ is a product of inverse limits of $p$-nilspaces and the decomposition follows. We can then  extend again Proposition \ref{prop:pmorphismsep} to the case of compact nilspaces and use it to prove the uniqueness of the decomposition. We omit the details.

\section{A $U^3$ inverse theorem in the $m$-torsion setting}\label{sec:proof-of-main}
\noindent In this section we prove Theorem \ref{thm:mainU3}. Given an $m$-torsion finite abelian group $\ab$ and a 1-bounded function $f$ on $\ab$ with $\|f\|_{U^3}\ge \delta>0$, our starting point is the conclusion of Theorem \ref{thm:invreduc-intro}, namely that $f$ correlates with a function of the form $\prod_{p\in \mc{P}}F_p\co \phi_p(z_p)$, where $\mc{P}=\mc{P}(|\ab|)$ and $\ab=\bigoplus_{p\in \mc{P}}\ab_p$ is the primary decomposition of $\ab$ (so each $z\in\ab$ is uniquely written as  $(z_p)_{p\in \mc{P}}$, with $z_p\in \ab_p$). Our strategy to prove Theorem \ref{thm:mainU3} is then  basically to establish a positive answer to Question \ref{mainQ:p-case} for the $U^3$-norm. More precisely, fixing any prime $p\in \mc{P}$, we shall prove that for the above morphism $\phi_p:\ab_p\to\ns_p$, there is another morphism from $\ab_p$ to a 2-step \emph{abelian} group nilspace, which \emph{refines} $\phi_p$ as per the following terminology from \cite{CGSS}.
\begin{defn}\label{def:refined-morphism}
Given two maps $\psi_i:A\to B_i$, $i=1,2$ defined on all of $A$ (but with $B_1$, $B_2$ possibly different spaces), we say that $\psi_2$ \emph{refines} $\psi_1$, and write $\psi_1\lesssim \psi_2$, if the partition generated by $\psi_2$ refines the partition generated by $\psi_1$, i.e.\ if the partitions $\mc{P}_i:=\{\psi_i^{-1}(y):y\in\psi_i(A)\}$, $i=1,2$ satisfy that every set in $\mc{P}_1$ is a union of some sets in $\mc{P}_2$.    
\end{defn}
\noindent Note that clearly $\psi_2$ refines $\psi_1$ if and only if there is a map $h:\psi_2(A)\to \psi_1(A)$ such that $\psi_1=h\co \psi_2$ (namely the map $h$ taking each $x=\psi_2(a)\in \psi_2(A)$ to $\psi_1(a)$).

Thus, our plan to prove Theorem \ref{thm:mainU3} is to show that there exists a finite abelian group $G$, of order bounded in terms of $|\ns_p|$ only, a degree-2 filtration $G_\bullet$ on $G$, and a morphism $\psi:\ab_p\to G$ (where $G$ is given the group nilspace structure associated with $G_\bullet$), such that $\phi\lesssim \psi$. 

The main tool that we will use is the following theorem, which generalizes (in a weaker form) a ``quadratic Hahn-Banach" type result known in the $\mb{F}_p^n$ setting (see e.g.\ \cite[Exercise 11.2.7]{Tao-Vu}).
\begin{theorem}\label{thm:mainextthm} For any positive integers $n,r,e$ and any prime $p$, there exists\footnote{An explicit bound for $c$ in terms of $n,r,e,p$ can be deduced from our arguments, but this is not needed here.} $c=c(n,r,e,p)>0$ such that the following holds. Let $k\in \mb{N}$, let $A$ be a $p^e$-torsion finite abelian group and let $H\leq A$ be a subgroup of index $r$. Let $B$ be a finite abelian group of order $n$, and let $\phi:\mathcal{D}_1(H)\to\mathcal{D}_k(B)$ be a nilspace morphism. Then there is an abelian group $C$ of order at most $c$ and a morphism $\psi:\mathcal{D}_1(A)\to\mathcal{D}_k(C)$ such that $\phi\lesssim\psi|_H$. 
\end{theorem}
\begin{remark}
Generalizing quadratic Hahn-Banach results (such as \cite[Exercise 11.2.7]{Tao-Vu}) from finite vector spaces to $p^e$-torsion abelian groups \emph{requires} some weakening such as the one in Theorem \ref{thm:mainextthm}, where the new map $\psi$ extends the old map $\phi$ only in the weaker sense that $\psi$ \emph{refines} $\phi$ as per Definition \ref{def:refined-morphism}. Indeed, it can be shown that a direct extension in this more general bounded-torsion setting is not always possible. For instance, for any odd prime $p$, let $G=\mb{Z}_p\times\mb{Z}_{p^2}$, let $H$ be the subgroup $\mb{Z}_p\times p\mb{Z}_{p^2}\leq G$, and let  $\phi:H\to\mb{T}$, $(x,py)\mapsto \tfrac{2}{p}xy$ (where $\mb{T}=\mb{R}/\mb{Z}$). It can be checked that $\phi$ is a well-defined quadratic map on $H$ but that it cannot be extended to a globally quadratic phase function on $G$. For if it could be extended then by \cite[Lemma 3.1]{GT08} the extension $\phi':G\to\mb{T}$ should be of the form $\phi'(x,y)=\langle M(x,y),(x,y)\rangle+\langle (\xi_1,\xi_2),(x,y)\rangle+\theta$ for some self-adjoint homomorphism $M:G\to G$ and elements $\xi\in G$ and $\theta\in \mb{T}$, and where $\langle \cdot,\cdot\rangle:G\times G\to \mb{T}$ is the symmetric non-degenerate bilinear form  $\langle (x,y),(x'y')\rangle:=\frac{1}{p}xx'+\frac{1}{p^2}yy'$ (recall \cite[Definition 4.1]{Tao-Vu}). From this we could then deduce, letting $M(1,0)=(a,pb)$ and $M(0,1)=(c,d)$ for integers $a,b,c,d$, that $\phi'(x,y)=\frac{1}{p}ax^2+\frac{1}{p^2}pbxy+\frac{1}{p}cxy+\frac{1}{p^2}dy^2+\frac{1}{p}\xi_1 x+\frac{1}{p^2}\xi_2 y+\theta$. Then, since we must have  $\phi'(x,py)=\tfrac{2}{p}xy$, evaluating at $x=0$ and $y=0$ separately we would deduce  $\phi'(x,py)=\frac{1}{p}ax^2+\frac{1}{p}\xi_1 x$. Evaluating this again in $y=0$, we would conclude that $\phi'$ is identically 0-valued on $H$, which is impossible since $\phi$ clearly is not identically 0. (Note that \cite[Proposition 3.2]{GT08} holds on $\mb{F}_p^n$ but not in its stated level of generality.)
\end{remark}

\noindent We split the proof of Theorem \ref{thm:mainextthm} into several steps. 

Recall that given a finite abelian group $A$ and a subgroup $H\le A$, we say that $H$ has a \emph{complement in $A$} if there exists a subgroup $K\le A$ such that $K+H=A$ and $K\cap H=\{0\}$.

\begin{proposition}\label{prop:comp1} Let $p$ be a prime number, and let $n,r$ be positive integers. 
Let $A$ be a $p^n$-torsion abelian group and let $H\leq A$ be a subgroup of index $r$. Then there is a subgroup $H'\leq H$ of index at most $r^{n^2+n}$ such that $H'$ has a complement in $A$. 
\end{proposition}

The proof of this proposition will use the following lemma.
\begin{lemma}\label{lem:larger-has-complement} Let $A$ be a $p^n$-torsion finite abelian group and suppose that $H\leq A$ is a subgroup generated by at most $r$ elements. Then there is a subgroup $H'$ with $H\leq H'\le A$ such that $|H'|\leq p^{n^2r}$ and $H'$ has a complement in $A$.
\end{lemma}

To prove this we will use the following fact about abelian $p$-groups.
\begin{lemma}\label{lem:enlarge-for-complement}
Let $p$ be a prime, let $n$ be a positive integer, and let $A$ be a $p^n$-torsion finite abelian group. Then for every $x\in A$ there exists a subgroup $H\le A$ of order at most $p^{n^2}$ such that $x\in H$ and $H$ has a complement in $A$.
\end{lemma}

\begin{proof}
If $x=0$ then the result is trivial so we can assume that $x\not=0$.

Without loss of generality, we can replace $x$ with an element $x'\in A$ such that $x\in \langle x'\rangle$ and $x'$ has no $p$-th roots. By assumption we have  $A=\prod_{i=1}^\ell \mb{Z}_{p^{a_i}}$ where $a_i\le n$ for all $i\in[\ell]$ by definition of $A$. Rearranging these factors $\mb{Z}_{p^{a_i}}$ if necessary, we can assume that for every $x'=(x_1',\ldots,x_\ell')\in \prod_{i=1}^\ell \mb{Z}_{p^{a_i}}$ the first ${r^{(1)}}$ components $x_1',\ldots,x_{r^{(1)}}'$ are coprime with $p$ and the remaining components $x'_{{r^{(1)}}+1},\ldots,x'_\ell$ are multiples of $p$. Note that ${r^{(1)}}\ge 1$, otherwise $x'$ would have a $p$-th root. We now decompose $A$ as $A^{(1)}\times \{0^{\ell-{r^{(1)}}}\}\oplus \{0^{r^{(1)}}\}\times B^{(1)}$ where $A^{(1)}:=\prod_{i=1}^{r^{(1)}} \mb{Z}_{p^{a_i}}$ and $B^{(1)}:=\prod_{i={r^{(1)}}+1}^\ell \mb{Z}_{p^{a_i}}$. Now consider $x^{(1)}:=(x_1',\ldots,x_{r^{(1)}}')\in A^{(1)}$ and $y^{(1)}:=(x'_{{r^{(1)}}+1},\ldots,x'_\ell)\in B^{(1)}$. Clearly $x'=(x^{(1)},0^{\ell-{r^{(1)}}})+(0^{{r^{(1)}}},y^{(1)})$. Then $\langle x^{(1)}\rangle$ is a cyclic subgroup of maximal order inside $A^{(1)}$ and thus by known results ( e.g.\ \cite[p. 44, 2.1.2]{KSfintie groups}) the subgroup $\langle x^{(1)}\rangle$ of $A^{(1)}$ has a complement in $A^{(1)}$.

Note now that if $y^{(1)}=0\in B^{(1)}$ (this happens also in particular if $r^{(1)}=\ell$) then we can take $H:=\langle (x^{(1)},0^{\ell-r^{(1)}})\rangle\le A$ and the result follows (the complement of $H$ would then be $C^{(1)}\times B^{(1)}$). Hence we can assume that $y^{(1)}\in B^{(1)}$ is not zero, and since all its coordinates are multiples of $p$, we deduce that $y^{(1)}$ has a non-zero $p$-th root ${y^{(1)}}'\in B^{(1)}$.

Now we repeat the argument with ${y^{(1)}}'\in B^{(1)}$. That is, first we replace ${y^{(1)}}'$ with some element which has no $p$-th root. Abusing the notation, let us denote this element by ${y^{(1)}}'$. Then we decompose $B^{(1)}=\prod_{i=r^{(1)}+1}^\ell \mb{Z}_{p^{a_i}}$ in two parts:  $B^{(1)}=A^{(2)}\times \{0^{\ell-r^{(1)}-r^{(2)}}\}\oplus \{0^{r^{(2)}}\}\times B^{(2)}$ where $A^{(2)}:=\prod_{i=r^{(1)}+1}^{r^{(2)}} \mb{Z}_{p^{a_i}}$ and $B^{(2)}:=\prod_{i=r^{(1)}+r^{(2)}+1}^\ell \mb{Z}_{p^{a_i}}$ such that there are two elements $x^{(2)}\in A^{(2)}$ and $y^{(2)}\in B^{(2)}$ satisfying ${y^{(1)}}' = (x^{(2)},0^{\ell-r^{(1)}-r^{(2)}})+(0^{r^{(2)}},y^{(2)})$ where $\langle x^{(2)} \rangle$ has a complement $C^{(2)}$ in $A^{(2)}$ and all the coefficients of $y^{(2)}$ are multiples of $p$.

We repeat the above argument $t$ times until we reach the point where in the decomposition we have $y^{(t)}=0$ (or, equivalently, in the decomposition of $A$ as a direct sum, the second term is trivial). Note that at each step we have forced that all the coefficients of the non-zero element $y^{(1)},\ldots,y^{(t-1)}$ are multiples of $p$. Since all these elements consist of $n$ components (being elements of the original group $A=\prod_{i=1}^\ell \mb{Z}_{p^{a_i}}$) we have  $t\le n$. 

Let $H:=\langle (x^{(1)},0^{\ell-r^{(1)}}),\ldots,(0^{r^{(1)}+\cdots+r^{(t-1)}},x^{(t)},0^{\ell-r-\cdots-r^{(t)}}) \rangle$, and note that $x\in H$ and $|H|\le p^{n^2}$. Moreover, this group has a complement in $A$. Namely, the subgroup
\[
C:=(\{0^{r^{(1)}+\cdots+r^{(t)}}\}\times B^{(t)})\oplus \bigoplus_{i=1}^t \left(\{0^{r^{(1)}+\cdots+r^{(i-1)}}\}\times C^{(i)}\times \{0^{\ell-r^{(1)-\cdots-r^{(i)}}}\}\right)
\]
can be checked to satisfy $C+H=A$ and $C\cap H=\{0\}$. The result follows.
\end{proof}

\begin{proof}[Proof of Lemma \ref{lem:larger-has-complement}]
We argue by induction on $r$. If $r=1$ then $H=\langle x\rangle$ for some $x$. We have that $x$ is contained in a subgroup $H'$ which has a complement by Lemma \ref{lem:enlarge-for-complement}. We have that $H\leq H'$ and that $|H'|\leq p^{n^2}$ and so the statement of the lemma holds. Let us assume that the statement holds for $r-1$ and that $H=\langle x_1,x_2,\dots,x_r\rangle$. Let $G=\langle x_1,x_2,\dots,x_{r-1}\rangle$. We have by induction that there is a subgroup $G'$ containing $G$ such that $G'$ has a complement $K$ and $|G'|\leq p^{n^2(r-1)}$. Let us write $x_r=k+g$ where $k\in K$ and $g\in G'$. By Lemma \ref{lem:enlarge-for-complement} we have that $k$ is contained in a subgroup $C$ of $K$ such that $C$ has a complement $K'$ inside $K$ and $|C|\leq p^{n^2}$. We have that $K'$ is a complement of $C+G'$ in $A$  and $|C+G'|\leq |C||G'|\le  p^{n^2r}$. On the other hand it is clear that $H\leq C+G'$.
\end{proof} 

\begin{proof}[Proof of Proposition \ref{prop:comp1}]
Let $t$ be the rank of $A/H$, and note that since $A/H$ is also an abelian $p$-group we have $t\leq\log_p(|A:H|)\le \log_p(r)$. We can choose a subgroup $T$ of $A$ generated by $t$ elements such that $H+T=A$. (Here $T$ can be chosen to be the subgroup generated by the preimages of a generating system of $A/H$ under the homomorphism $A\to A/H$.) Let $Q=H\cap T$. Note that $Q$ is generated by at most $t$ elements. Now by Lemma \ref{lem:larger-has-complement} we can embed $Q$ into some subgroup $Q'\leq H$ with $|Q'|\leq p^{tn^2}$ such that $Q'$ has a complement $K$ in $H$. Note that $K$ has the subgroup $Q'+T$ as a complement in $A$, indeed $K+Q'+T=H+T=A$ and by construction $K\cap (Q'+T)=K\cap Q'=\{0\}$. Thus $[A:K]=|Q'+T|\leq |Q'||T|\leq p^{tn^2+tn}\leq r^{n^2+n}$ and $K$ is a good choice for $H'$. 
\end{proof}

We are ready to prove the main result of this section.

\begin{proof}[Proof of Theorem \ref{thm:mainextthm}] By Proposition \ref{prop:comp1} the group $H\leq A$ has a subgroup $H'$ of bounded index such that $H'$ has a complement $K$ in $A$. In particular, for each coset $S$ of $H'$ in $H$ there is a nilspace morphism $\tau_S:\mc{D}_1(A)\to\mc{D}_1(S)$ such that $\tau_S$ is the identity map on $S$ ($\tau_S$ can be thought of as a projection $A\to S$). Indeed, every $x\in A$ can be uniquely written as $x=s+k$ where $s\in S,k\in K$, and we then define $\tau_S(x):=s$. Now we can define $\psi$ to be the map $x\mapsto (\phi\circ\tau_S(x)\big)_{\textrm{cosets $S$ of $H'$ in $H$}}$ which goes from $\mathcal{D}_1(A)$ to $\mathcal{D}_k(C)$ where $C$ is the direct power $B^{|H:H'|}$. Note that $\psi|_H$ clearly refines $\phi$, as every $x\in H$ is in some coset $H'$ and so $\phi(x)$ will be one of the components of $\psi(x)$.
\end{proof}

\begin{proof}[Proof of Theorem \ref{thm:mainU3}]
By Theorem \ref{thm:invreduc-intro} for $k=2$, we have that $f$ correlates with a function of the form $\prod_{p\in \mc{P}}F_p\co \phi_p(z_p)$ where $\ab\cong \prod_{p\in \mc{P}} \ab_p$ is the $p$-Sylow decomposition of $\ab$ (in particular $\mc{P}=\mc{P}(|\ab|)$), where each $\phi_p$ is a morphism from $\mc{D}_1(\ab_p)$ to some 2-step finite $p$-nilspace $\ns_p$.

Fix any prime $p\in\mc{P}$. Note that $\pi_1\co\phi_p:\mc{D}_1(\ab_p)\to \mc{D}_1(\ab_1(\ns_p))$ is an affine homomorphism where $\ab_1(\ns_p)$ is the first structure group of $\ns_p$. Let $H_p\le \ab_p$ be the kernel of $\pi_1\co\phi_p-\pi_1\co\phi_p(0)$. The index of $H_p\le \ab_p$ is bounded by the size of $\ab_1(\ns_p)$, which is in turn bounded in terms of $\delta$ and the torsion $m$ only.

Fix a complete set of coset representatives $\{z_i\in \ab_p:i\in I_p\}$ for $H_p$ in $\ab_p$, i.e.\ such that $\bigsqcup_{i\in I_p} z_i+H_p$ is a partition of $\ab_p$. Note that $|I_p|\le |\ab_p:H_p|=O_{\delta,m}(1)$. For each $i\in I_p$, the map $\phi_i^{(p)}=\phi_p|_{z_i+H}:\mc{D}_1(H_p)\to \ns_p$ defined as $x\mapsto \phi_p(x+z_i)$ takes values by definition in a single $\pi_1$ fiber. Thus $\phi_i^{(p)}$ is a morphism $\mc{D}_1(H_p)\to \mc{D}_2(\ab_2(\ns_p))$. Note that in particular $|\ab_2(\ns_p)|\le |\ns_p|=O_{\delta,m}(1)$. For each $i\in I_p$ we apply Theorem \ref{thm:mainextthm}, obtaining a function $\psi_i^{(p)}:\mc{D}_1(\ab_p)\to \mc{D}_2(C_i^{(p)})$ with $|C_i^{(p)}|=O_{\delta,m}(1)$ such that $\phi_i^{(p)}\lesssim \psi_i^{(p)}$.

Let $\psi_p:\mc{D}_1(\ab_p)\to\mc{D}_1(\ab_1(\ns_p))\times  \prod_{i\in I}\mc{D}_2(C_i^{(p)})$ be defined as $x\mapsto (\pi_1(\phi_p(x)), (\psi_i^{(p)}(x+z_i))_{i\in I_p})$. This map is a morphism and  $\phi_p\lesssim \psi_p$.

Having thus shown that the answer to Question \ref{mainQ:p-case} is affirmative in this case $k=2$, the result follows essentially by Fourier analysis, i.e.\ by Corollary \ref{cor:Q1.9reduc}. 
\end{proof}

\begin{remark}
One might have hoped to use, instead of Lemma \ref{lem:enlarge-for-complement}, a claim that if an element $x$ of a finite abelian $p$-group has no $p$-th root, then the group $\langle x\rangle$ has a complement (this would imply that every element $y$ in an abelian $p$-group is contained in some cyclic subgroup $\langle x\rangle$ that has a complement). However, this claim is false. Indeed, let $A:=\mb{Z}_3\times \mb{Z}_{27}$ and $x:=(1,3)\in A$. Then $x$ has no 3-rd root in $A$ and yet $\langle x\rangle$ has no complement (we omit the proof). Nevertheless, Lemma \ref{lem:enlarge-for-complement} gives us that $\langle x\rangle\le H=A$, which does have a complement (the trivial group $\{(0,0)\}$) and size $|H|=3^4\le 3^{3^2}$.
\end{remark}
\noindent Let us record the following straightforward consequence of Proposition \ref{prop:comp1}, which can be useful more generally in the $m$-torsion setting. 
\begin{corollary}\label{cor:mtorocomp}
For any $r,m\in \mb{N}$ there exists $c\in\mb{N}$ such that the following holds. For any $m$-torsion finite abelian group $A$, and any subgroup $H\leq A$ of index $r$, there is a subgroup $H'\leq H$ of index at most $c$ such that $H'$ has a complement in $A$.
\end{corollary}
\noindent This is obtained using the primary decomposition of $A$ and applying Proposition \ref{prop:comp1} in each $p$-group component; we omit the details.
\begin{remark}\label{rem:AltPfs}
Once a result such as Proposition \ref{prop:comp1} or Corollary \ref{cor:mtorocomp} is proved, there are in fact several ways to apply this tool to prove Theorem \ref{thm:mainU3}. In particular, after the first preprint version of this paper appeared online, it was observed by Jamneshan, Shalom and Tao (in personal communication) that one can also obtain Theorem \ref{thm:mainU3} by combining such a tool with \cite[Theorem 1.6]{J&T}. Indeed, in the $m$-torsion setting the Bohr set obtained in \cite[Theorem 1.6]{J&T} is essentially a bounded-index subgroup, inside which Corollary \ref{cor:mtorocomp} enables us to find a \emph{complemented} bounded-index subgroup, and from the latter subgroup it is then possible to extend the quadratic polynomial $\phi$ (obtained in \cite[Theorem 1.6]{J&T}) to the whole group, to obtain the desired correlating global quadratic phase function. This alternative proof yields an explicit bound for $\varepsilon$ in Theorem \ref{thm:mainU3}. 
\end{remark}

\section{The $U^{k+1}$ inverse theorem with projected phase polynomials}\label{sec:prophase}
\noindent In this section we first prove Theorem \ref{thm:invboundedtor}. We will use several ingredients from \cite{SzegFin}. The first one is  the following (see \cite[Lemma 2.5]{SzegFin}).
\begin{lemma}\label{lem:perio-1}
For any positive integers $k,\ell$ there exists  $\alpha>0$ such that the following holds. For any $m$-torsion abelian group $A$, any morphism $\phi:\mc{D}_\ell(\mb{Z})\to \mc{D}_k(A)$ is $m^{\alpha}$-periodic.
\end{lemma}

\begin{proof}
By \cite[Lemma A.2]{CGSS-doublecoset}, the morphism $\phi$ is in fact a morphism $\mc{D}_1(\mb{Z})\to \mc{D}_{\lfloor k/\ell\rfloor}(A)$, so we may assume that $\ell=1$. Now by \cite[Theorem A.6]{CGSS-doublecoset}, any morphism $\phi$ of the latter type has an expression of the form $\phi(x)=\sum_{i=1}^{\lfloor k/\ell\rfloor}a_i\binom{x}{i}$ for some $a_i\in A$. Let us prove now that $\binom{x+m^{\lfloor k/\ell\rfloor+1}}{i}-\binom{x}{i}$ is a multiple of $m$ for any $x\in \mb{Z}$ and $i\in[\lfloor k/\ell\rfloor]$. For any prime $p|m$ suppose that $m=p^cm'$ where $p$ and $m'$ are coprime. If we prove that $\binom{x+m^{\lfloor k/\ell\rfloor+1}}{i}-\binom{x}{i}$ is a multiple of $p^c$ then we are done (as we can then argue similarly for every prime dividing $m$).

Using the identity $\binom{x}{i}=\frac{x(x-1)\cdots(x-i+1)}{i}$ we have $\binom{x+m^{\lfloor k/\ell\rfloor+1}}{i}-\binom{x}{i}=\frac{m^{\lfloor k/\ell\rfloor+1}}{i!}Q(x,m,i)$ for some integer-valued polynomial $Q$. If we prove that $\frac{m^{\lfloor k/\ell\rfloor+1}}{i!}$ is always a multiple of $p^c$ then we will be done. Note that the largest power of $p$ dividing $m^{\lfloor k/\ell\rfloor+1}$ is precisely $c(\lfloor k/\ell\rfloor+1)$. On the other hand, in $i!$ we have at most $\sum_{j=1}^\infty \lfloor i/p^j\rfloor\le \sum_{j=1}^\infty i/p^j=\frac{i}{p-1}\le \frac{\lfloor k/\ell\rfloor}{p-1}$ factors of $p$. But as for any $c\in \mb{N}$ and $p$ prime we have that $\frac{\lfloor k/\ell\rfloor}{p-1}+c\le c(\lfloor k/\ell\rfloor+1)$ the result follows.
\end{proof}
\noindent Next, we extend Lemma \ref{lem:perio-1}, letting $\mc{D}_k(A)$  be replaced by finite nilspaces \cite[Lemma 2.6]{SzegFin}.
\begin{lemma}\label{lem:perio-2}
For any positive integers $k,\ell$ there exists $\beta>0$ such that the following holds. Let $m\in \mb{N}$, let $\ns$ be a $k$-step $m$-torsion finite nilspace, and let $\phi:\mc{D}_\ell(\mb{Z})\to \ns$ be a morphism. Then $\phi$ is $m^\beta$-periodic.
\end{lemma}

\noindent Note that given any finite $k$-step nilspace $\ns$ we can always find some $m=m(|\ns|)$ such $\ns$ is $m$-torsion. Thus the result also holds for \emph{any} finite $k$-step nilspace $\ns$, the conclusion then being that $\phi$ is $m^{\beta}$-periodic with $m$ depending on $|\ns|$.

\begin{proof}
We argue by induction on $k$. For $k=1$ the result follows from Lemma \ref{lem:perio-1}. Then in particular, if $\pi_{k-1}:\ns\to \ns_{k-1}$ is the projection to the $k-1$ factor, we have that there exists $\beta_{k-1}$ such that $\pi_{k-1}\co \phi:\mc{D}_\ell(\mb{Z})\to \ns_{k-1}$ is $m^{\beta_{k-1}}$ periodic. In particular, for any fixed $i\in[m^{\beta_{k-1}}]$ we have that $\phi(i+xm^{\beta_{k-1}})$ (as a function of $x\in \mb{Z}$) is in $\hom(\mc{D}_\ell(\mb{Z}),\mc{D}_k(\ab_k(\ns)))$ where $\ab_k(\ns)$ is the $k$-th structure group of $\ns$. By Lemma \ref{lem:perio-1} we deduce that $\phi(i+xm^{\beta_{k-1}})$ is $m^{\alpha}$ periodic for some $\alpha=\alpha(k,\ell)$ and thus $\phi$ is $m^{\alpha+\beta_{k-1}}$ periodic.
\end{proof}
\noindent We quickly deduce the following corollary (see \cite[Theorem 7]{SzegFin}). In order to phrase it, following \cite[Definition 1.2]{CGSS-doublecoset} we recall that a \emph{discrete free nilspace} $F$ is a nilspace of the form $\prod_{i=1}^k \mc{D}_i(\mb{Z}^{a_i})$ for some $k\in\mb{N}$ and $a_i\in\mb{Z}_{\ge 0}$, $i\in[k]$. For any other nilspace $\ns$ and any morphism $\phi:F\to \ns$ we will say that $\phi$ is $r$-periodic  if it factorizes through the map $p:F=\prod_{i=1}^k \mc{D}_i(\mb{Z}^{a_i})\to \prod_{i=1}^k \mc{D}_i(\mb{Z}^{a_i}_r)$ defined as $(x_{i,j})_{i\in[k],j\in [a_i]}\in F\mapsto (x_{i,j}\mod r)_{i\in[k],j\in [a_i]}\in \prod_{i=1}^k \mc{D}_i(\mb{Z}^{a_i}_r)$. That is, there exists a morphism $\tilde{\phi}:\prod_{i=1}^k \mc{D}_i(\mb{Z}^{a_i}_r)\to \ns$ such that $\phi = \tilde{\phi}\co p$.

\begin{corollary}\label{cor:periodicfactor}
For any $k\in\mb{N}$ there exists $\gamma>0$ such that the following holds. Let $m$ be a positive integer, let $\ns$ be a finite $k$-step $m$-torsion nilspace, and let $F$ be a $k$-step discrete free nilspace. Let $\phi:F\to \ns$ be a morphism. Then $\phi$ is $m^\gamma$-periodic. 
\end{corollary}

\begin{proof}
Let $F=\prod_{i=1}^k\mc{D}_i(\mb{Z}^{a_i})$ for some integers $a_i\in \mb{Z}_{\ge0}$, $i\in[k]$. For $i\in[k]$ and $j\in[a_i]$ let $e_{i,j}\in F$ be the element with all coordinates equal to 0 except for the $j$th coordinate of $\mc{D}_i(\mb{Z}^{a_i})$ where it equals 1. For any $x\in F$ and $r\in \mb{Z}$ note that the map $\phi(x+re_{i,j})$ (as a function of $r\in \mb{Z}$) is in $\hom(\mc{D}_i(\mb{Z}),\ns)$. By Lemma \ref{lem:perio-2} there exists $\beta=\beta(k,i)$ such that $\phi(x+re_{i,j})$ is $m^\beta$ periodic. Thus, if we let $\gamma(k):=\max_{i\in[k]}(\beta(k,i))$ we deduce that $\phi$ is $m^\gamma$ periodic. In particular, $\phi$ factorizes through $p$.
\end{proof}
\noindent Combining this with a central  result from \cite{CGSS-doublecoset}, we obtain a useful fact on finite nilspaces.
\begin{corollary}\label{cor:fin-ns-quot-modN}
Let $m\ge 1$ be an integer. Let $\ns$ be a finite $k$-step nilspace such that all its structure groups have torsion $m$. Then there exists $\gamma=\gamma(k)\in\mb{N}$ and  integers $a_i\in \mb{Z}_{\ge 0}$, $i\in[k]$, such that for the group nilspace $\nss=\prod_{i=1}^k\mc{D}_i(\mb{Z}_{m^\gamma}^{a_i})$ there is a fibration $\varphi:\nss \to \ns$.
\end{corollary}
\begin{proof}
We apply \cite[Theorem 4.4]{CGSS-doublecoset}, obtaining a $k$-step free nilspace $F$ and a continuous fibration $\varphi':F\to \ns$. The continuity of $\varphi'$ implies (factoring through possible continuous components of $F$ if necessary) that we may assume without loss of generality that $F$ is a \emph{discrete} free nilspace, thus $F=\prod_{i=1}^k \mc{D}_i(\mb{Z}^{a_i})$ for some $a_i\in\mb{Z}_{\geq 0}$. We now apply Corollary \ref{cor:periodicfactor}, letting $\gamma$ be the resulting integer, obtaining the claimed fibration $\varphi:\nss\to\ns$ as the map with $\varphi'=\varphi\co p$.
\end{proof}
\noindent We can give a quick proof of the inverse theorem with projected phase polynomials.
\begin{proof}[Proof of Theorem \ref{thm:invboundedtor}]
We apply Theorem \ref{thm:nil-for-m-exp} and let $\ns$ be the resulting nilspace of torsion $m$, $\phi$ the resulting morphism $\mc{D}_1(\ab)\to\ns$, and $F:\ns\to\mb{C}$ the resulting 1-bounded function such that $\mb{E}_{x\in Z} f(x) F(\phi(x))\geq \delta^{2^{k+1}}/2$. By Corollary \ref{cor:fin-ns-quot-modN} there is a fibration $\varphi:\nss\to \ns$. By standard group theory there is a surjective group homomorphism $\tilde\tau:\mb{Z}^r\to \ab$, where $r$ is the rank of $\ab$. In particular $\phi\co\tilde\tau$ is a nilspace morphism $\mc{D}_1(\mb{Z}^r)\to \ns$. By \cite[Corollary A.6]{CGSS-p-hom}, there is a morphism $g:\mc{D}_1(\mb{Z}^r)\to \nss$ such that $\varphi\co g=\phi\co\tilde\tau$. By a very special case of Corollary \ref{cor:periodicfactor} (applied with $\ns$ equal to the nilspace $\nss$ here, and $F$ there equal to $\mc{D}_1(\mb{Z}^r)$) we have that $g$ is actually $m^\gamma$-periodic, so we can replace $\mb{Z}^r$ by an $m$-torsion abelian group $B=\mb{Z}_{m^\gamma}^r$, replace $g$ by a morphism $\psi:\mc{D}_1(B)\to\nss$, and $\tilde\tau$ by a surjective homomorphism $\tau:B\to \ab$. The situation can be summarized with the following commutative diagram:
\begin{equation}\label{diag:proj-phases}
\begin{tikzpicture}
  \matrix (m) [matrix of math nodes,row sep=2em,column sep=3em,minimum width=2em]
  { \mc{D}_1(\mb{Z}^r)    &    &  \\
     & \mc{D}_1(B) & \nss   \\
     & \ab & \ns \\};
  \path[-stealth]
(m-2-3) edge node [right] {$\varphi$} (m-3-3)
(m-1-1) edge node [above] {$g$} (m-2-3)
(m-1-1) edge node [below] {$\tilde\tau$} (m-3-2)
(m-1-1) edge node [below] {$p$} (m-2-2)
(m-3-2) edge node [above] {$\phi$} (m-3-3)
(m-2-2) edge node [right] {$\tau$} (m-3-2)
(m-2-2) edge node [below] {$\psi$} (m-2-3);
\end{tikzpicture}
\end{equation}
Let $h$ denote the function $F\co \varphi:\nss\to\mb{C}$. Then $\mb{E}_{x\in Z} f(x) F(\phi(x))= \mb{E}_{y\in B} f\co\tau(y) h\co\psi(y)$. By the Fourier decomposition of $h$ on the finite abelian group $\nss$, and the pigeonhole principle, there is a character $\chi\in \wh{B}$ such that $\varepsilon\leq \mb{E}_{y\in B} f(\tau(y)) \chi(\psi(y)) = \mb{E}_{x\in Z} f(x) \mb{E}_{y\in \tau^{-1}(x)}\chi(\psi(y))$, which proves the result with $\phi:=\chi\co \psi$.
\end{proof}
\noindent The rest of this section treats the two aspects of Theorem \ref{thm:invboundedtor} mentioned in the introduction.

Recall that the first aspect is that the projected phase polynomials are ensured to have degree at most $k$ (in comparison with inverse theorems for $\|\cdot\|_{U^{k+1}}$ in which the degree of the polynomial is only bounded in terms of $k$ but may be \emph{larger} than $k$). Among other things, this implies that this inverse theorem is ``tight" in the sense that projected phase polynomials of degree $k$ are indeed obstructions to degree-$k$ uniformity (unlike phase functions of degree only guaranteed to be bounded in terms of $k$). More precisely, we have the following result.
\begin{proposition}\label{prop:propolyobstruct}
Let $\phi_{*\tau}$ be a projected phase polynomial of degree $k$ on a finite abelian group $\ab$, and suppose that $f:\ab\to\mb{C}$ satisfies $|\langle f,\phi_{*\tau}\rangle|\geq \delta$. Then $\|f\|_{U^{k+1}}\geq \delta$.
\end{proposition}
This follows quickly from the following fact of independent interest.
\begin{lemma}\label{lem:dualnormbound}
Let $\phi_{*\tau}$ be a projected phase polynomial of degree $k$ on a finite abelian group. Then $\|\phi_{*\tau}\|_{U^{k+1}}^* \le  1$ where $\|\cdot \|_{U^{k+1}}^*$ is the $U^{k+1}$-dual-norm.
\end{lemma}
\begin{proof}
Recall the definition $\|\phi_{*\tau}\|_{U^{k+1}}^*=\sup_{g:\ab\to\mb{C}:\|g\|_{U^{k+1}}\leq 1}|\langle \phi_{*\tau},g\rangle|$. Denoting by $\ab'$ the (abelian group) domain of $\tau$, the map $\tau^{\db{k+1}}:\cu^{k+1}(\ab')\to \cu^{k+1}(\ab)$ defined by $\tau^{\db{k+1}}(\q):v\mapsto \tau(\q(v))$ is a surjective homomorphism. It follows that for every map $g:\ab\to\mb{C}$ we have $\| g\co \tau \|_{U^{k+1}(\ab')} = \|g\|_{U^{k+1}(\ab)}$. Then we have $|\langle \phi_{*\tau}, g\rangle_{\ab} |= $ $| \mb{E}_{x\in \ab} \mb{E}_{y\in \tau^{-1}(x)} g\co\tau(y) \phi(y) | = | \langle g\co\tau, \phi\rangle_{\ab'}|\leq \|g\co\tau\|_{U^{k+1}(\ab')} \|\phi\|_{U^{k+1}(\ab')}^* = \|g\|_{U^{k+1}(\ab)} \|\phi\|_{U^{k+1}(\ab')}^*$. Therefore $\|\phi_{*\tau}\|_{U^{k+1}(\ab)}^*\leq \|\phi\|_{U^{k+1}(\ab)}^*$. Since $\phi$ is a phase polynomial of degree $k$, we have that $|\langle \phi,g\rangle|=\|\phi \overline{g}\|_{U^1}\le \|\phi \overline{g}\|_{U^{k+1}}= \| g\|_{U^{k+1}}$, see \cite[(2.1)]{GT08}. Thus $\|\phi\|_{U^{k+1}(\ab')}^*\le 1$.
\end{proof}
\begin{proof}[Proof of Proposition \ref{prop:propolyobstruct}]
By Lemma \ref{lem:dualnormbound}, $\delta\leq |\langle f, \phi_{*\tau}\rangle |\leq \|f\|_{U^{k+1}} \|\phi_{*\tau}\|_{U^{k+1}}^*\leq \|f\|_{U^{k+1}} $. 
\end{proof}
\noindent The second aspect was that Theorem \ref{thm:invboundedtor}, in addition to being tight in the sense explained above, is also stronger than inverse theorems in which the correlating harmonic is a phase polynomial of degree $C(m,k)$. We shall prove this in the following subsection.
\subsection{On a result of Jamneshan, Shalom and Tao}\hfill\\
The idea is to prove that the projected phase polynomials appearing in Theorem \ref{thm:invboundedtor} can be written as averages of polynomials of possibly larger degree. We shall prove this below and then apply it to give an alternative proof of \cite[Theorem 1.12]{JST-tot-dis}, which we recall here for convenience.
\begin{theorem}\label{thm:JST-tot-dis}
Let $m,k$ be positive integers and $\delta>0$. Then there exist $\varepsilon=\varepsilon(m,k,\delta)>0$ and $C=C(k,m)>0$ such that the following holds. For any finite abelian $m$-torsion group $\ab$ and any 1-bounded function $f:\ab\to \mb{C}$ with $\|f\|_{U^{k+1}}\ge \delta$ there exists a polynomial map $P:\ab\to \mb{R}/\mb{Z}$ of degree at most $C$ such that $|\langle f,e(-P)\rangle|\ge \varepsilon$.
\end{theorem}
\noindent Given a surjective homomorphism $\tau:B\to\ab$, by a \emph{polynomial cross-section} for $\tau$ we mean a map $\iota:\ab\to B$ which is polynomial and such that $\tau\co \iota$ is the identity map on $\ab$.

The main result that we will use is the following.
\begin{theorem}\label{thm:most-general-cross-poly}
Let $m,m'\ge 1$ be integers. Then there exists a constant $C(m,m')\in \mb{N}$ such that the following holds. Let $\ab,B$ be finite abelian groups of torsion $m$ and $m'$ respectively and let $\tau:B\to \ab$ be a surjective homomorphism. Then there exists a polynomial cross-section $\iota:\ab\to B$ of degree at most $C(m,m').$\end{theorem}

\begin{remark}
Not every cross-section is a polynomial map. For example, let $\tau:\mb{Z}_6\to \mb{Z}_3$ be the map $x\mod 6\mapsto x\mod 3$. The cross-section $\iota:\mb{Z}_3\to \mb{Z}_6$ defined as $0\mapsto 0, 1\mapsto 1$ and $2\mapsto 5$ can be proved not to be a polynomial map of any degree.
\end{remark}

We split the proof of Theorem \ref{thm:most-general-cross-poly} into several lemmas.

\begin{lemma}\label{lem:thm:poly-cross-sec-1}
Let $d\ge s$ be positive integers and let $p$ be a prime. Let $\varphi:\mb{Z}_{p^d}\to \mb{Z}_{p^s}$ be the map $x\mod p^d\mapsto x\mod p^s$. Let $\iota:\mb{Z}_{p^s}\to \mb{Z}_{p^d}$ be  defined by $n\!\mod p^s\mapsto n\!\mod p^d$ for each $n\in [0,p^s-1]$. Then $\iota$ is a polynomial cross-section for $\varphi$ of degree at most $(d-s)p^s+1$.
\end{lemma}

\begin{remark}
Note that this result follows from \cite[Lemma 8.2]{JST-tot-dis} with degree-bound $d(p^s-1)$, which is sometimes better and sometimes worse than our bound, depending on $d,s$ and $p$.
\end{remark}
\noindent The argument has similarities with the proof of \cite[Proposition B.2]{CGSS-p-hom}. We want to prove that if we take sufficiently many derivatives of $\iota$ then we obtain the 0 map. Without loss of generality, it suffices to take derivatives $\partial_a \iota(x):=\iota(x+a)-\iota(x)$ with respect to the generator $a=1\in \mb{Z}_{p^s}$. Note that $\partial_1 \iota(x)=1$ if $x\not=p^s-1$ and $\partial_1 \iota(p^s-1)=1-p^s$. Taking one more derivative, $\partial^2_1 \iota(x)=0$ if $x\not=p^s-1$, $\partial^2_1 \iota(p^s-2)=-p^s$ and $\partial^2_1 \iota(p^s-1)=p^s$. To take derivatives of higher degree, as is standard, we can view the map $\partial_1^2 \iota$ as a vector in $\mb{Z}^{p^s}_{p^d}$ and take the derivatives by left-multiplying this vector by the forward difference matrix, i.e.\ the circulant matrix  $C_{p^s}:=\begin{psmallmatrix} -1 & 1 & 0 & \cdots & 0 \\
0& -1 & 1  & \cdots & 0 \\
\vdots &   &   & \ddots & \\
1 & 0 & \cdots & 0 & -1\end{psmallmatrix}\vspace{0.05cm}\in M_{p^s\times p^s}(\mb{Z})$. Known results on circulant matrices imply the following fact.
\begin{lemma}\label{lem:powers-A_P^s}
For any prime $p$ and any integer $s\ge 1$ all the entries of $C_{p^s}^{p^s}$ are multiples of $p$.
\end{lemma}
\begin{proof}
By equation (8) in \cite{Feng}, for every $q\in\mb{N}$ we have $C_{p^s}^q=\sum_{j=0}^q\binom{q}{j}(-1)^{j}A_{p^s}^{q-j}$, where $A_{p^s}$ is the cyclic permutation matrix (see \cite{Feng}). Taking $q=p^s$, we claim that it suffices to prove that $\binom{p^s}{j}=\frac{p^s!}{j!(p^s-j)!}$ is a multiple of $p$ if $0<j<p^s$. In fact, the contributions for $j=0$ and $j=p^s$ cancel each other if $p$ is odd as $A_{p^s}^0=A_{p^s}^{p^s}=\id_{p^s\times p^s}$ and thus $\binom{p^s}{0}(-1)^0\id_{p^s\times p^s}+\binom{p^s}{p^s}(-1)^{p^s}\id_{p^s\times p^s} = 0$. If $p=2$ we have $\binom{2^s}{0}(-1)^0\id_{2^s\times 2^s}+\binom{2^s}{2^s}(-1)^{2^s}\id_{2^s\times 2^s} = 2\id_{2^s\times 2^s}$ which is a multiple of $p=2$ as claimed. To see the general case $0<j<p^s$, note first that the number of $p$ factors in $j!$ is precisely $\sum_{i=1}^{s-1} \lfloor j/p^i\rfloor$. Thus, it suffices to prove that
$\sum_{i=1}^{s-1} \lfloor j/p^i\rfloor+\sum_{i=1}^{s-1} \lfloor (p^s-j)/p^i\rfloor < 1+p+\cdots+p^{s-1} = \frac{p^s-1}{p-1}$, 
where the right hand side is the number of $p$ factors of $p^s!$. The left hand side can be estimated using the bound $\sum_{i=1}^{s-1} \lfloor j/p^i\rfloor \le \sum_{i=1}^{s-1} j/p^i = j\frac{p^{s-1}-1}{(p-1)p^{s-1}}$. Hence, the left  side is bounded above by $j\frac{p^{s-1}-1}{(p-1)p^{s-1}}+(p^s-j)\frac{p^{s-1}-1}{(p-1)p^{s-1}} = \frac{p^s-p}{p-1}$, which is smaller than the number of $p$ factors in $p^s!$.
\end{proof}

\begin{proof}[Proof of Lemma \ref{lem:thm:poly-cross-sec-1}] Note that after two derivatives, the map $\partial^2_1 \iota$ has already a factor $p^s$. Each time that we differentiate $p^s$ additional times we add (at least) a factor $p$ by Lemma \ref{lem:powers-A_P^s}. It follows that $\partial^{kp^s+2}_1 \iota(x)$ is a multiple of $p^{k+s}$ for any $x\in \mb{Z}_{p^s}$. Thus, if $k+s=d$ then we have $\partial^{kp^s+2}_1 \iota=0 \!\mod p^d$. Hence $\iota$ is a polynomial of degree at most $(d-s)p^s+1$.
\end{proof}
\begin{remark}
The degree estimate in Lemma \ref{lem:thm:poly-cross-sec-1} is not tight. For example, let $\varphi:\mb{Z}_{3^2}\to \mb{Z}_3$ be the map $x\mod 9\mapsto x \mod 3$. Then $\iota:\mb{Z}_3\to \mb{Z}_{3^2}$ is a polynomial of degree 3 (this can be chekced manually). However, Lemma \ref{lem:thm:poly-cross-sec-1} gives the bound $(2-1)3^1+1=4$.
\end{remark}

\noindent We shall want to apply Lemma \ref{lem:thm:poly-cross-sec-1} to surjective homomorphisms on more general groups. The idea is that a general surjective homomorphism $\tau:B\to\ab$ can be decomposed according to the $p$-Sylow subgroups of $B$ and $\ab$, and this reduces the task to the case where $B,\ab$ are $p$-groups. To handle this case we shall use the following technical result.
\begin{proposition}\label{prop:technical-prop-poly-cross}
Let $p$ be a prime number and $n\in \mb{N}$. Let $B,\ab$ be finite abelian groups of torsion $p^n$ and let $M:B\to \ab$ be a surjective homomorphism. Then there exist the following:\setlength{\leftmargini}{0.7cm}
\begin{enumerate}
    \item some $m\in \mb{Z}_{\ge 0}$ and some abelian group $\ab'$ of torsion $p^{n-1}$ such that $\ab \cong\ab' \times \mb{Z}_{p^n}^m$,
    \item an abelian group $B'$ of torsion $p^{n-1}$ and a  surjective homomorphism $A:B'\to \ab'$,
    \item a surjective homomorphism $P:B\to B'\times \mb{Z}_{p^n}^m$ with the following property: There exists some isomorphisms   $B\to \prod_{i=1}^n \mb{Z}_{p^{i}}^{a_i}$ for\footnote{Note that since $M$ is surjective we have $a_n\ge m$.} $a_i\in \mb{Z}_{\ge 0}$  and $B'\to  (\prod_{i=1}^{n-1} \mb{Z}_{p^{i}}^{a_i}) \times \mb{Z}_{p^{n-1}}^{a_n-m}$   such that, abusing the notation, if we consider $P$ as a map $P:\prod_{i=1}^n \mb{Z}_{p^{i}}^{a_i}\to (\prod_{i=1}^{n-1} \mb{Z}_{p^{i}}^{a_i})\times \mb{Z}_{p^{n}}^{m} \times \mb{Z}_{p^{n-1}}^{a_n-m} $ (via these isomorphisms), then $P$ is the identity map in every coordinate except in the first $m$ terms of $\mb{Z}_{p^n}^{a_n}$, where it equals the natural projection $\mb{Z}_{p^n}\to \mb{Z}_{p^{n-1}}$, 
    \item  isomorphisms $S:\ab\to \ab$ and $T:B\to B$,
\end{enumerate}
such that $M=S\co (A,\id_{\mb{Z}_{p^n}^m})\co P \co T$.
\end{proposition}

\begin{proof}
Let $B:=\prod_{i=1}^n \mb{Z}_{p^{i}}^{a_i}$ and $\ab:=\prod_{i=1}^n \mb{Z}_{p^{i}}^{b_i}$ for some integers $a_i,b_i\ge 0$, $i\in[n]$. For $n=1$ the result follows by linear algebra.

We now assume that $n\ge 2$. For $i\in[n]$ and $j\in [b_i]$ let $\pi_{i,j}:\prod_{i=1}^n \mb{Z}_{p^{i}}^{b_i}\to \mb{Z}_{p^i}$ the projection to the $j$-th coordinate of the term $\mb{Z}_{p^i}^{b_i}$. For any such $i,j$, note that $\pi_{i,j}\co M$ is a surjective homomorphism, so letting $\pi_{i,j}\co M((x_{\ell,k})_{\ell\in[n],k\in[a_\ell]}) = \sum_{\ell=1}^n\sum_{k=1}^{a_\ell} r_{\ell,k}^{(i,j)}x_{\ell,k}$ for some $r_{\ell,k}^{(i,j)}\in \mb{Z}$,  there exists some $k=k_j\in[a_n]$ such that $r_{n,k}^{(n,j)}$ is coprime with $p$. Composing with an appropriate isomorphism $T_1:B\to B$ (just swapping two coordinates), we can assume that $M\co T_1$ is such that $\pi_{n,1}\co M\co T_1((x_{\ell,k})_{\ell\in[n],k\in[a_\ell]}) = x_{n,1}+\sum_{(\ell,k)\not=(n,1)}r_{\ell,k}^{(i,j)}x_{\ell,k}$. Then, composing with an appropriate isomorphism $S_1:\ab\to \ab$, we can assume that $S_1\co M\co T_1$ is such that in the expression of $\pi_{i,j}\co S_1\co M\co T_1((x_{\ell,k})_{\ell\in[n],k\in[a_\ell]})$ for $(i,j)\not=(n,1)$ there is no coefficient of $x_{n,1}$. Composing further with another isomorphism $T_1':B\to B$, we can assume that in addition to the previous properties, we have that $\pi_{n,1}\co S_1\co M\co T_1\co T_1'((x_{\ell,k})_{\ell\in[n],k\in[a_\ell]}) = x_{n,1}$.

We now repeat this process for the coordinate $\pi_{n,2}\co S_1\co M\co T_1\co T_1'((x_{\ell,k})_{\ell\in[n],k\in[a_\ell]})$, then for $\pi_{n,3}\co S_1\co M\co T_1\co T_1'((x_{\ell,k})_{\ell\in[n],k\in[a_\ell]})$, and so on, until we obtain that there are isomorphisms $S:\ab\to \ab$ and $T:B\to B$ such that $\pi_{n,j}\co S\co M\co T((x_{\ell,k})_{\ell\in[n],k\in[a_\ell]})=x_{n,j}$ for all $j\in[b_n]$. In particular, letting $P:B\to (\prod_{i=1}^{n-1} \mb{Z}_{p^{i}}^{a_i})\times \mb{Z}_{p^{n}}^{b_n}\times \mb{Z}_{p^{n-1}}^{a_n-b_n}$ be the map that projects the last $a_n-b_n$ coordinates of $\mb{Z}_{p^n}^{a_n}$ from $\mb{Z}_{p^n}$ to $\mb{Z}_{p^{n-1}}$ there exists a surjective homomorphism $A:(\prod_{i=1}^{n-1} \mb{Z}_{p^{i}}^{a_i})\times \mb{Z}_{p^{n-1}}^{a_n-b_n}\to (\prod_{i=1}^{n-1} \mb{Z}_{p^{i}}^{b_i})$ such that, letting $(A,\id_{\mb{Z}_{p^n}^{b_n}}): (\prod_{i=1}^{n-1} \mb{Z}_{p^{i}}^{a_i})\times \mb{Z}_{p^n}^{b_n}\times \mb{Z}_{p^{n-1}}^{a_n-b_n}\to (\prod_{i=1}^{n-1} \mb{Z}_{p^{i}}^{b_i})\times \mb{Z}_{p^n}^{b_n}\cong \ab$ be the natural map, then $S\co M\co T = (A,\id_{\mb{Z}_{p^n}^{b_n}})\co P$. The result follows.
\end{proof}

\begin{proof}[Proof of Theorem \ref{thm:most-general-cross-poly}] Any such homomorphism $\tau:B\to \ab$ can be split into the components defined by the $p$-Sylow decomposition of $\ab$ and $B$. In fact, let $\Lambda$ be the set of primes that divide either $m$ or $m'$. Then there exists isomorphisms $\ab\cong \prod_{p\in \Lambda} \ab_p$ and $B\cong \prod_{p\in \Lambda} B_p$ where $\ab_p,B_p$ are $p$-groups. Abusing the notation and considering $\tau:\prod_{p\in \Lambda} \ab_p\to \prod_{p\in \Lambda} B_p$ there exists homomorphisms $\tau_p:\ab_p\to B_p$ for $p\in \Lambda$ such that $\tau((x_p\in \ab_p)_{p\in \Lambda})=(\tau_p(x_p))_{p\in \Lambda}$. Clearly all the torsion of the groups $\ab_p,B_p$ for $p\in \Lambda$ divide $m$ and $m'$ respectively. Hence, if Theorem \ref{thm:most-general-cross-poly} holds for $p$-groups of bounded torsion we can find cross-sections $\iota_p:B_p\to \ab_p$ of $\tau_p$ for $p\in \Lambda$ and the map $\iota:=(\iota_p)_{p\in \Lambda}:\prod_{p\in \Lambda} \ab_p\to \prod_{p\in \Lambda} B_p$ is clearly a polynomial cross-section of $\tau$ with degree bounded in terms of $m$ and $m'$.

Thus, it suffices to prove the result assuming that both $\ab$ and $B$ are $p$-groups for some prime $p$. Let $n$ be such that $B$ has torsion $p^n$. We will prove the result by induction on $n$, with the case $n=1$ following from linear algebra.

For $n\ge 2$, by Proposition \ref{prop:technical-prop-poly-cross} we have that $\tau = S\co (A,\id_{\mb{Z}_{p^n}^m})\co P\co T$ where $S$ and $T$ are isomorphisms and $A$ and $P$ are surjective homomorphisms. We want to show that there exists cross-sections of the maps $S,(A,\id_{\mb{Z}_{p^n}^m}),P$ and $T$ which are polynomials of some bounded degree (depending only on $p$ and $n$).

Clearly as $S$ and $T$ are isomorphisms their inverses are polynomial cross-sections of degree 1. By $(iii)$ of Proposition \ref{prop:technical-prop-poly-cross} combined with Lemma \ref{lem:thm:poly-cross-sec-1} there exists a cross-section $\iota_1:B'\times \mb{Z}_{p^n}^m\to B$ such that $P\co \iota_1=\id_{B'\times \mb{Z}_{p^n}^m}$ of degree at most $O_{p,n}(1)$. Finally, by $(i)$ and $(ii)$ of Proposition \ref{prop:technical-prop-poly-cross} we have that $A$ is a surjective homomorphism between $p$-groups of torsion at most $p^{n-1}$. Hence, by induction on $n$ there exists a polynomial cross-section $\iota_2:\ab'\to B'$ such that $A\co \iota_2=\id_{\ab'}$ of degree at most $O_{n,p}(1)$. By \cite[Lemma A.2]{CGSS-doublecoset} it follows that $T^{-1}\co \iota_1\co (\iota_2,\id_{\mb{Z}_{p^n}^m})\co S^{-1}$ is a polynomial cross-section of $\tau$ of degree at most $O_{n,p}(1)$.
\end{proof}

We can now give the main proof of this subsection.

\begin{proof}[Proof of Theorem \ref{thm:JST-tot-dis}]
By Theorem \ref{thm:invboundedtor}, the function $f$ correlates with a projected phase polynomial  $(\chi\co \psi)_{*\tau}$ of degree $k$ and torsion $(m,m^{O_k(1)})$, for some homomorphism $\tau:B\to \ab$. By Theorem \ref{thm:most-general-cross-poly} there exists a polynomial cross-section $\iota:\ab\to B$ of degree $O_{m,k}(1)$ such that $\tau\co \iota=\id_{\ab}$. Moreover, for any $u\in \ker(\tau)$ we have that $\iota_u(x):=\iota(x)+u$ is clearly also a polynomial cross-section. Recall that $(\chi\co \psi)_{*\tau}$ is the map  $\mb{E}_{y\in \tau^{-1}(x)}\chi(\psi(y))$ defined for $x\in \ab$. However, for any $x\in \ab$ we have $\mb{E}_{y\in \tau^{-1}(x)}\chi(\psi(y))=\mb{E}_{u\in \ker(\tau)}\chi(\psi(\iota_u(x)))$. Thus $\varepsilon < |\mb{E}_{x\in \ab} f(x)\mb{E}_{u\in \ker(\tau)}\chi(\psi(\iota_u(x)))| = |\mb{E}_{u\in \ker(\tau)}\mb{E}_{x\in \ab} f(x)\chi(\psi(\iota_u(x)))|$. Hence, $\varepsilon < |\mb{E}_{x\in \ab} f(x)\chi(\psi(\iota_u(x)))|$ for some $u\in \ker(\tau)$. Finally note that by \cite[Lemma A.2]{CGSS-doublecoset}, $\psi\co \iota_u$ is in fact a polynomial map with degree bounded by $\deg(\iota)k=O_{m,k}(1)$. The result follows.
\end{proof}

\section{Final remarks and open questions}
\noindent In this section we briefly discuss possible routes towards further progress on  \cite[Question 1.9]{JST-tot-dis}.

Given Theorem \ref{thm:invboundedtor}, one approach is to examine if a projected phase polynomial of degree $k$ can always be approximated by a bounded number of true phase polynomials of degree $k$.

\begin{question}\label{projconweak}
For any positive integers $m,m'$ and any $\eta>0$, does there always exist a positive integer $R=R(m,m',\eta)$ such that the following holds? Let $\ab,\ab'$ be finite abelian groups with torsions $m,m'$ respectively, and let $\phi_{*\tau}$ be a projected phase polynomial of degree $k$ on $\ab$, where $\tau:\ab'\to \ab$ is a surjective  homomorphism. Then there are phase polynomials $\phi_1, \phi_2, \ldots, \phi_R$ of degree $k$ on $\ab$, and 1-bounded coefficients $\lambda_i\in\mb{C}$, $i\in[R]$, such that $\big\|\phi_{*\tau}-\sum_{i=1}^R \lambda_i\phi_i\big\|_1\leq\eta$.
\end{question}

\noindent We note that this question is in fact an alternative formulation of \cite[Question 1.9]{JST-tot-dis}. More precisely, we have the following equivalence.
\begin{proposition}
A positive answer to Question \ref{projconweak} is equivalent to a positive answer to \cite[Question 1.9]{JST-tot-dis}.
\end{proposition}

\begin{proof}
To see the forward implication, we start from the conclusion of Theorem \ref{thm:invboundedtor}, thus we have a projected phase polynomial $\phi_{*\tau}$ such that $|\langle f,\phi_{*\tau}\rangle|\ge \varepsilon>0$ for some $\varepsilon=\varepsilon(\delta,m,k)$. Applying the positive answer to Question \ref{projconweak} with $\eta:=\varepsilon/2$ we obtain $R=R(m,m^{\gamma})$, polynomials $\phi_i$ for $i\in [R]$ and complex numbers $\lambda_i$, $i\in[R]$. Then $\varepsilon \le |\langle f,\phi_{*\tau}\rangle| \le |\langle f,\phi_{*\tau}-\sum_{i=1}^R \lambda_i\phi_i+\sum_{i=1}^R \lambda_i\phi_i\rangle| \le \|f\|_\infty\|\phi_{*\tau}-\sum_{i=1}^R \lambda_i\phi_i\|_1+ |\langle f,\sum_{i=1}^R \lambda_i\phi_i\rangle|\le \varepsilon/2+R\max_{i\in [R]}|\langle f,\phi_i\rangle|$, so there is $j\in [R]$ such that $|\langle f,\phi_j\rangle|\geq \varepsilon/(2R)$.

To see the converse, we can use a standard argument involving the Hahn-Banach theorem from work of Gowers--Wolf (see e.g.\ \cite{GHB}). Indeed, suppose there is no decomposition of the form $\phi_{*\tau}=\sum_{i=1}^R \lambda_i f_i +g$ with $\sum_{i=1}^R |\lambda_i| \leq R$, with $f_i$ being phase polynomials of degree $k$ on $\ab$, and with $g:\ab\to \mb{C}$ satisfying $\|g\|_1\leq \varepsilon$. Then, as explained in \cite[page 25]{GHB} (see the paragraph \emph{Consequences of failure of decomposition}), by the Hahn-Banach theorem there exists $\psi:\ab\to\mb{C}$ satisfying the following properties: firstly $\langle f, \psi\rangle\leq R^{-1}$ for every phase function $f:\ab\mapsto \mb{C}$ of degree $k$, secondly $\langle \psi, \phi_{*\tau}\rangle >1$, and finally $\|\psi\|_\infty \leq \varepsilon^{-1}$. By Lemma \ref{lem:dualnormbound} we have $\|\phi_{*\tau}\|_{U^{k+1}}^*\leq 1$. This together with the second property above implies that $\|\psi\|_{U^{k+1}}> 1$. Hence the function $\varepsilon\psi$ has $U^{k+1}$-norm greater than $\varepsilon$ and has $L^\infty$-norm at most 1, so by the positive answer to \cite[Question 1.9]{JST-tot-dis} there is a  phase function $f:\ab\mapsto \mb{C}$ of degree $k$ such that $\langle f, \varepsilon\psi\rangle > \delta(\varepsilon)$. Therefore $\langle f, \psi\rangle > \varepsilon^{-1} \delta$, which contradicts the first property above if $R= \varepsilon\, \delta^{-1}$.
\end{proof}
\noindent Let us emphasize again here that a positive answer to \cite[Question 1.9]{JST-tot-dis} would also follow from a positive answer to Question \ref{mainQ:p-case}.

Finally, recall the connection between Gowers norms and generalized cut norms established by Corollary \ref{cor:box-norm-estimate}. The norms involved in that result have rather elementary definitions, and yet to prove the result we made a heavy use of nilspace theory. This prompts the following question.
\begin{question}
Is there an elementary proof of Corollary \ref{cor:box-norm-estimate}?
\end{question}

\end{document}